%% file: BBM3Problems.tex
\documentclass[a4paper,11pt]{article}

\usepackage{mathtools}
\usepackage{amsmath}
\usepackage{amssymb}
\usepackage{amsthm}
\usepackage{enumitem}
\usepackage{graphicx}
\usepackage{mathabx}
\usepackage{bbm}
\usepackage{verbatim}
\usepackage{color}
\usepackage[%
ocgcolorlinks,%
linkcolor=blue,%
filecolor=blue,%
citecolor=blue,%
urlcolor=blue]{hyperref}
\usepackage{sidecap}
\usepackage{enumitem}

\usepackage[font={small,it}]{caption}

\newcommand{\bbL}{\mathbb{L}}

\newcommand{\bbM}{\mathbb{M}}
\newcommand{\bbE}{\mathbb{E}}

\newcommand{\bbP}{\mathbb{P}}

\newcommand{\bbR}{\mathbb{R}}

\newcommand{\eps}{\epsilon}

\newcommand{\wh}{\widehat}
\newcommand{\wc}{\widecheck}

\newcommand{\cA}{\mathcal A}

\newcommand{\cC}{\mathcal C}
\newcommand{\cX}{\mathcal X}

\newcommand{\cB}{\mathcal B}

\newcommand{\cW}{\mathcal W}

\newcommand{\cE}{\mathcal E}
\newcommand{\cN}{\mathcal N}
\newcommand{\cF}{\mathcal F}

\newcommand{\cR}{\mathcal R}
\newcommand{\cQ}{\mathcal Q}

\newcommand{\rmc}{{\rm c}}

\newcommand{\rmd}{{\rm d}}
\newcommand{\rme}{{\rm e}}
\newcommand{\rmB}{{\rm B}}

\newcommand{\rmE}{{\rm E}}

\newcommand{\wt}{\widetilde}

\newtheorem{theorem}{Theorem}[section]
\newtheorem{cor}[theorem]{Corollary}
\newtheorem{lem}[theorem]{Lemma}
\newtheorem{prop}[theorem]{Proposition}
\newtheorem{thm}[theorem]{Theorem}

\newtheorem*{rem*}{Remark}
\newtheorem{rem}[theorem]{Remark}

\numberwithin{equation}{section}
\graphicspath{{pic/}}

\newif\ifusebibtex

\newif\ifonlystable
\onlystabletrue 

\ifonlystable
\title
{A Probabilistic Proof for Stable Fluctuations in the Extremal Process of Branching Brownian Motion}
\else
\title
{On the Fluctuations of the Extremal and Cluster Level Sets of Branching Brownian Motion}
\fi
\date{}
\author
{Lisa Hartung \thanks{lhartung@uni-mainz.de,  oren.louidor@gmail.com, wtq0000@gmail.com} \\ 
Universit\"at Mainz\and Oren Louidor\footnotemark[1]\\Technion, Israel\and 
Tianqi Wu\footnotemark[1]\\Technion, Israel }

\begin{document}
\maketitle
\begin{abstract}
\ifonlystable
We give a probabilistic proof for the emergence of the Stable-$1$ Law for the random fluctuations of the mass of the extremal process of branching Brownian Motion away from its tip. This result was already 
shown by Mytnik et al.~\cite{MRR22} albeit using PDE techniques. As a consequence, we demystify the origin of these fluctuations and the meaning of the deterministic centering function required.
\fi
\end{abstract} 

\setcounter{tocdepth}{2}
\tableofcontents

\section{Introduction and results}
\input{intro}

\section{Preliminaries}
\label{s:2}
\input{preliminaries}

\input{stable}

\appendix
\section{Appendix: Proofs for preliminary statements}

\subsection{Joint convergence of critical extreme level sets and maximum}
\label{s:A3}
\input{critical}

\ifonlystable
\subsection{Ballot Estimates}
\fi
\label{s:A2}
\input{ballot}

\section*{Acknowledgments}
\ifonlystable
We are grateful to L. Mytnik and L. Ryzhik for useful discussions.
\fi
The research of L.H.~was supported by the Deutsche Forschungsgemeinschaft (DFG, German Research Foundation)  through Project-ID 233630050 - TRR 146, through Project-ID 443891315  within SPP 2265, and Project-ID 446173099. The research of O.L.~and T.W.~was supported by the ISF grant no.~2870/21, and by the BSF award 2018330.

\ifusebibtex
	\bibliographystyle{abbrv}
	\bibliography{BBM3Problems}
\else

\input{BBM3Problems.bbl}
\fi
\end{document}

%% file: intro.tex
\subsection{Setup and background}
\ifonlystable
This work concerns fluctuations in the extremal process of branching Brownian motion (BBM) away from the tip.
\fi
Let us begin by recalling the definition of BBM and some of the theory concerning its extreme value statistics.
 Let $L_t$ be the set of particles alive at time $t \geq 0$ in a continuous-time Galton-Watson process with binary branching at rate $1$. The entire genealogy can be recorded via the metric space $(T,\rmd)$, consisting of the elements $T := \bigcup_{t \geq 0} L_t$ and equipped with the {\em genealogical distance} 
\begin{equation}
\rmd (x,x') := \tfrac12 \Big(\big(|x|-|x \wedge x'|\big) + \big(|x'|-|x \wedge x'|\big)\Big)
\, , \qquad x,x' \in T \,.
\end{equation}
In the above, $|x|$ stands for the generation of $x$, namely $t$ such that $x \in L_t$ and $x \wedge x'$ is the most recent common ancestor of $x$ and $x'$. We shall also write $x_s$ for the ancestor of $x$ at generation $0 \leq s \leq |x|$.

Conditional on $(T,\rmd)$, let $h = (h(x) :\: x \in T)$ be a mean-zero continuous Gaussian process with covariance function given by 
\begin{equation}
	\bbE h(x) h(x') = |x \wedge x'| \, , \qquad x,x' \in T \,.
\end{equation}
The triplet $(h, T,\rmd)$ (or just $h$ for short) forms a standard BBM, and $h(x)$ for $x \in L_t$ is interpreted as the height of particle $x$ at time $t$. As such, this process models a system in which particles diffuse as standard Brownian motions (BMs) and undergo binary splitting at $\text{Exp}(1)$-distributed random times. The restriction of $T$ to all particles born up to time $t$ will be denoted by $T_t := \bigcup_{s \leq t} L_s$, with $\rmd_t$ and $h_t$ the corresponding restrictions of $\rmd$ and $h$, respectively. The natural filtration of the process $(\cF_t :\: t \geq 0)$ can then be defined via $\cF_t = \sigma(h_t, T_t, \rmd_t)$ for all $t \geq 0$.

The study of extreme values of $h$ dates back to the works of Ikeda et al.~\cite{Ikeda1,Ikeda2,Ikeda3}, McKean~\cite{McKean}, Bramson \cite{B_M, B_C} and Lalley and Sellke~\cite{LS} who derived asymptotics for the law of the
maximal height $h^*_t = \max_{x \in L_t} h_t(x)$. Introducing the centering function
\begin{equation}
\label{eq_m_t}
m_t := \sqrt{2}t - \frac{3}{2\sqrt{2}} \log^+ t \,,
\quad  \text{where } \qquad
\log^+ t := \log (t \vee 1) \,,
\end{equation}
and writing
$\wh{h}_t$ for the centered process $h_t - m_t$ and $\wh{h}^*_t : = h_t^\ast - m_t$ for the centered maximum, these works show 
\begin{equation}
\label{e:101.4}
\wh{h}^*_t \,\underset{t \to \infty}\Longrightarrow \,
\wh{h}^*_\infty := G + \tfrac{1}{\sqrt{2}} \log Z  \,,
\end{equation}
where $G$ is a Gumbel distributed random variable with rate $\sqrt{2}$ and $Z$, which is independent of $G$, is the almost sure limit as $t \to \infty$ of (a multiple of) the so-called {\em derivative martingale}:
\begin{equation}
\label{e:303}
\textstyle
Z_t := C_\diamond \sum_{x \in L_t} \big( \sqrt{2} t - h_t(x) \big) \rme^{\sqrt{2} (h_t(x) - \sqrt{2}t)} \,,
\end{equation}
for some properly chosen $C_\diamond > 0$. 

Other extreme values of $h$ can be studied simultaneously by considering its {\em extremal process}. To describe the latter, given $t \geq 0$, $x \in L_t$ and $r > 0$, we let $\cC_{t,r}(x)$ denote the {\em cluster} of {\em relative heights} of particles in $L_t$, which are at genealogical distance at most $r$ from $x$. This is defined formally as the point measure
\begin{equation}
\label{e:5B}
\textstyle
\cC_{t,r}(x) := \sum_{y \in \rmB_{r}(x)} \delta_{h_t(y) - h_t(x)},\  \text{ where} \quad
\rmB_{r}(x) := \{y \in L_t :\: \rmd (x,y) < r\} \,.
\end{equation}

Fixing any positive function $[0,\infty) \ni t \mapsto r_t$ such that both $r_t$ and $t-r_t$ tend to $\infty$ as $t \to \infty$ and letting
\begin{equation}
	L_t^* = \big\{x \in L_t :\: h_t(x) \geq h_t(y) \,, \forall y \in \rmB_{r_t} (x)\big\} \,,
\end{equation} 
the {\em structured extremal process} is then given as
\begin{equation}
\label{e:N6}
\textstyle
\wh{\cE}_t := \sum_{x \in L_t^*} \delta_{h_t(x) - m_t} \otimes \delta_{\cC_{t,r_t}(x)} \,.
\end{equation}
That is, $\wh{\cE_t}$ is a point process on $\bbR \times \bbM_p((-\infty,0])$, where $\bbM_p((-\infty,0])$ denote the space of  Radon point measures on $(-\infty, 0]$, which records the {\em centered-heights-of} and the {\em clusters-around} all $r_t$-local maxima of $h|_{L_t}$. Then, it was (essentially) shown in~\cite{ABBS2013, ABK_E} that
\begin{equation}
\label{e:N7}
\big(\wh{\cE}_t,\,\wh{h}_t^*, Z_t \big) \underset{t \to \infty}{\Longrightarrow} \big(\wh{\cE},\, \wh{h}_\infty^*,\, Z\big)\,.
\end{equation}
Above $Z$ and $\wh{h}^*_\infty$ are as in~\eqref{e:101.4}, and the law of $\wh{\cE}$ conditional on $Z$ is 
\begin{equation}
\label{e:A1.10}
\wh{\cE}\, |\, Z \sim {\rm PPP}(Z\rme^{-\sqrt{2}u} \rmd u \otimes \nu) 
\overset{\rmd}={\rm PPP}(\rme^{-\sqrt{2}(u - \frac{1}{\sqrt{2}} \log Z)} \rmd u \otimes \nu) \,,
\,,
\end{equation}
with, importantly, the same $Z$ used to realize all the objects on the right hand side of~\eqref{e:N7}.
The measure $\nu$ in~\eqref{e:A1.10} is a (nonrandom) probability distribution over the space $\bbM_p((-\infty,0])$, which we call the {\em cluster distribution}. 

As two consequences, one gets the convergence of the {\em extremal process of local maxima}:
\begin{equation}
\cE^*_t \underset{t \to \infty}\Longrightarrow \cE^*  \,,
\end{equation}
where
\begin{equation}
\label{e:5.5}
\cE_t^* := \sum_{x \in L^*_t} \delta_{h_t(x) - m_t} 
\quad , \qquad 
\cE^* := \sum_{(u, \cC) \in \wh{\cE}} \, \delta_u 
\,\sim\, {\rm PPP}(Z\rme^{-\sqrt{2}u} \rmd u) \,,
\end{equation}
and the convergence of the {\em usual} extremal process of $h$:
\begin{equation}
\label{e:A1.13}
\cE_t \underset{t \to \infty}\Longrightarrow \cE  \,,
\end{equation}
where
\begin{equation}
\label{e:O.2}
\cE_t := \sum_{x \in L_t} \delta_{h_t(x) - m_t} 
\, , \qquad 
\cE
:= \sum_{k \geq 1} \, \cC^k(\cdot - u^k) \,,
\end{equation}
in which $u^1 > u^2 > \dots$ enumerate the atoms of $\cE^*$ and $(\cC^k)_{k \geq 1}$ are i.i.d.~chosen according to $\nu$ and independent of $\cE^*$. 

Since, conditional on $Z$, the intensity measure of $\cE^*$ is finite on $[-v, \infty)$ for every $v \in \bbR$ and tends to $\infty$ as $v \to \infty$, it is a standard fact that $\bbP(-|Z)$-a.s. and therefore also $\bbP$-a.s.,
\begin{equation}
\label{e:1.3}
	\frac{\cE^*([-v, \infty))}{\bbE \big(\cE^*([-v, \infty))\,\big|\,Z\big)} 
	= \frac{\cE^*([-v, \infty))}{\tfrac{1}{\sqrt{2}} Z \rme^{\sqrt{2}v}} 
	\underset{v \to \infty}{\longrightarrow} 1 \,.
\end{equation}

However, the asymptotics for $\cE([-v, \infty))$, arguably a more interesting quantity, are not a straightforward consequence of~\eqref{e:5.5} and~\eqref{e:O.2}. This is because the limiting process $\cE$ is now a superposition of i.i.d.~clusters $\cC$, and the law $\nu$ of the latter will determine the number of points inside any given set in the overall process. To address this question, a study of the moments of the level sets of a cluster $\cC \sim \nu$ was carried out in~\cite{CHL19}. In particular, it was shown that  
as $v \to \infty$,
\begin{equation}
\label{e:29}
\bbE \cC([-v, 0]) = C_\star \rme^{\sqrt{2} v} (1+o(1)) \,.
\end{equation}
for some $C_\star > 0$. This was then combined with~\eqref{e:5.5} and~\eqref{e:O.2} to derive 
\begin{equation}
\label{e:1.10}
\frac{\cE([-v, \infty))}{C_\star Z v \rme^{\sqrt{2} v}} \overset{\bbP}{\underset{v \to \infty}{\longrightarrow}} 1 \,.
\end{equation}
Notice that compared to the asymptotics for the level sets of $\cE^*$, there is an additional linear prefactor (in $v$) in that of $\cE$, which is due to the contribution of the clusters.

Sharper asymptotics for $\cE([-v,\infty))$ were later derived (among other results) in~\cite{MRR22} using PDE techniques. Theorem~1.4 in that work, applied to the function $\phi_0 = 1_{(-\infty,0]}$ gives (with constants adjusted),
\begin{equation}
\label{e:2.1}
	\frac{1}{\rme^{\sqrt{2} u}} \Big[\cE\big([-u,\infty)\big)- C_\star Z \rme^{\sqrt{2}u} \big( u  + \sqrt{2} \log u \big)\Big] + CZ
	\underset{u \to \infty} \Longrightarrow \tfrac{C_\star}{\sqrt{2}} \cR_Z\,,
\end{equation}
where $(\cR_t)_{t \geq 0}$ is a spectrally positive stable-$1$ process, with Laplace Transform 
\begin{equation}
\label{e:A1.25}
\rmE \rme^{-\lambda \cR_t} = \exp{\big(t  \lambda \log \lambda \big)} \,,
\end{equation}
which is independent of $Z$, and $C \in \bbR$ is some explicit constant. 
The analogous result for the \emph{unshifted} process $\cE_\circ$, defined as $\cE$  in~\eqref{e:O.2} only with 
$u^1 > u^2 > \dots$ now enumerating the atoms of 
\begin{equation}
\label{e:A101.1}
\cE^*_\circ \sim {\rm PPP}(\rme^{-\sqrt{2}u} \rmd u) \,,
\end{equation}
(without the Z prefactor) gives
\begin{equation}
\label{e:2.1a}
	\frac{1}{\rme^{\sqrt{2} u}} \Big[\cE_\circ \big([-u,\infty)\big)- C_\star \rme^{\sqrt{2}u} \big( u  + \sqrt{2} \log u \big)\Big] + C
	\underset{u \to \infty} \Longrightarrow \tfrac{C_\star}{\sqrt{2}} \cR_1\,.
\end{equation}
We note that~\cite{MRR22} also shows that $C_\star C_\diamond  = 1/\sqrt{\pi}$, so that remarkably one can obtain explicit constants in~\eqref{e:2.1} and~\eqref{e:2.1a}.

\subsection{Contribution}
As pointed out in~\cite{MRR22} (see discussion below Theorem 1.4 in that work), the purely analytic methods used to derive~\eqref{e:2.1} and~\eqref{e:2.1a} provide no explanation for the emergence of the $1$-stable law governing the fluctuations in the limit, nor for the logarithmic order of the second order term. We aim to elucidate the reasons for these two results, by an alternative, probabilistic derivation for the asymptotic growth of $\cE$. 

Toward this end, following~\cite{CHL19}, we define a {\em restricted} version of the limiting extremal process $\cE$ from~\eqref{e:A1.13} and its unshifted analog $\cE_\circ$ just defined. Fixing a subset $B \subseteq \bbR$ and letting $(u^k)_{k \geq 1}$, $(\cC^k)_{k \geq 1}$ be as in~\eqref{e:O.2}, we set
\begin{equation}
\cE (\; \cdot \;; B) := \sum_{k} \cC^k ( \cdot - u^k) 1_B(u^k)
\end{equation}
This process is the restriction of $\cE$ where only contributions form clusters whose maximum is in $B$ are included. The restricted version in the unshifted case is defined in the same way, only with $(u^k)_{k \geq 1}$ now enumerating the atoms of $\cE^*_\circ$.

Our main result is
\begin{thm}
\label{t:A1.5}
Let $x^+, x^-: \bbR_+ \to \bbR$ be functions which tend sub-linearly but arbitrarily slow to $\pm \infty$ respectively and abbreviate $u_* \equiv -\frac{1}{\sqrt{2}} \log u$. Then 
\begin{equation}
\label{e:A1.26}
\frac{1}{e^{\sqrt{2} u}} \bigg[ \cE_\circ \Big([-u, \infty)\;;\; u_* + [x^-_u,\, x^+_u]\Big) - \bbE\, \cE_\circ \Big([-u, \infty)\;;\; u_* + [x^-_u,\, 0]\Big) \bigg] - C_\circ \underset{u \to \infty}\Longrightarrow \tfrac{C_\star}{\sqrt{2}} \cR_1 \,,
\end{equation}
and
\begin{equation}
\label{e:A1.27}
\frac{1}{e^{\sqrt{2} u}} \bigg[ \cE_\circ \Big([-u, \infty)\;;\; u_* + [x^-_u,\, x^+_u]^\rmc\Big) - \bbE\,\cE_\circ \Big([-u, \infty)\;;\; u_* + (-\infty,\, x^-_u) \Big) \bigg] \underset{u \to \infty}{\overset{\bbP}{\longrightarrow}} 0 \,,
\end{equation}
where $C_\star$ is as in~\eqref{e:29} and $C_\circ \in \bbR$ is a universal constant.
\end{thm}

Combining them together, we immediately get
\begin{cor}
\label{c:A1.6}With $\cR_1$ as in~\eqref{e:A1.25} and $C_\circ \in \bbR$ is as in Theorem~\ref{t:A1.5}.
\begin{equation}
\label{e:A1.29}
\frac{1}{e^{\sqrt{2} u}} \bigg[
\cE_\circ \big([-u, \infty)\big) - \bbE\,\cE_\circ  \Big([-u, \infty)\;;\; \big[-u, -\tfrac{1}{\sqrt{2}} \log u\big)\Big) \bigg] - C_\circ\underset{u \to \infty}\Longrightarrow \tfrac{C_\star}{\sqrt{2}} \cR_1 \,,
\end{equation}
\end{cor}
A version of this corollary for the original process $\cE$ is
\begin{cor}
\label{c:A1.7}
With $\cR_t$ as in~\eqref{e:A1.25} and $C_\circ$ as in Theorem~\ref{t:A1.5},
\begin{equation}
\label{e:A1.30}
\frac{1}{e^{\sqrt{2} u}} \bigg[
\cE \big([-u, \infty)\big) - \bbE \Big[\cE\Big([-u, \infty)\;;\; \big[-u, -\tfrac{1}{\sqrt{2}} \log u\big)\Big) \, \Big|\, Z \Big]\bigg] - C_\circ Z \underset{u \to \infty}\Longrightarrow \tfrac{C_\star}{\sqrt{2}} \cR_Z \,.
\end{equation}
\end{cor}

Corollary~\ref{c:A1.6} and Corollary~\ref{c:A1.7} are the analogues of~\eqref{e:2.1a} and~\eqref{e:2.1}. The centering functions in the scaling limit are now meaningful, but not explicit as in~\cite{MRR22}. On the other hand, by comparing e.g.~\eqref{e:A1.29} with~\eqref{e:2.1a}, we obtain that for some $C \in \bbR$,
\begin{equation}
\frac{1}{C_\star \rme^{\sqrt{2}u}} \bbE\,\cE_\circ  \Big([-u, \infty)\;;\; \big[-u, -\tfrac{1}{\sqrt{2}} \log u\big)\Big)  = u  + \sqrt{2} \log u + C + o(1) \,,
\end{equation}
with $o(1) \to 0$ as $u \to \infty$.

\begin{rem} Unfortunately, available estimates on $\bbE\, \cC([-v,0])$, namely~\eqref{e:29}, can only be used to show that the means in~\eqref{e:A1.29} and~\eqref{e:A1.30} are $C_\star u\rme^{\sqrt{2}u}(1+o(1))$ and $C_\star Zu\rme^{\sqrt{2}u}(1+o(1))$, respectively, which give coarser results compared to those in~\cite{MRR22}. Obtaining better estimates via the probabilistic approach would require better control over various sources of errors in the analysis of $\bbE\, \cC([-v,0])$ and falls outside the scope of this work.
\end{rem}

Going back to Theorem \ref{t:A1.5}, statements~\eqref{e:A1.26} and~\eqref{e:A1.27} together show that the random fluctuations in $\cE_\circ([-u, \infty))$ come from the contribution of those clusters whose maximum is within $\Theta(1)$ around $u_* := -\frac{1}{\sqrt{2}} \log u$. Let us now explain how the 1-stable law emerges from these fluctuations, as can be seen from the proof of this theorem. 

\subsubsection{Heuristic for 1-stable fluctuations in $\cE_\circ([-u, \infty))$}\label{s:stable_heuristic} 

At $u_x := u_* + x$ the intensity of $\cE_\circ^\star$ is $u\rme^{-\sqrt{2}x}$, so that the recentered version of $\cE_\circ^*$ has law
\begin{equation}
\label{e:A1.33a}
	\cE^*_u := \cE^*_\circ(\cdot - u_*) \sim {\rm PPP}\big(u \rme^{-\sqrt{2}x} \rmd x\big) \,.
\end{equation}
At the same time, clusters whose maximum is at height $u_x$  contribute a quantity according to the law of $\cC([-v_x , 0])$ to $\cE_\circ([-u, \infty))$, where $v_x := u+u_x = u - (1/\sqrt{2}) \log u + x$.

Now, while the mean of $\cC([-v,0])$ is $(C_\star + o(1)) \rme^{\sqrt{2}v}$ as stated in~\eqref{e:29}, it was shown in~\cite[Proposition~1.1]{CHL20} that this mean is the result of a $\Theta(v^{-1})$-probability event under which $\cC([-v,0]$ is atypically, $\Theta(v \rme^{\sqrt{2}v})$-large. A more precise version of this statement is given in Theorem~\ref{stable_key_lemma} of the present work 
and constitutes the main effort in the proof. The theorem shows that there exists
an infinite measure $\Gamma$ on $\bbR_+$, which is finite away from zero, such that
	\begin{equation}
	\label{e:A101.35}
		v \bbP\bigg(\frac{\cC([-v,0])}{v \rme^{\sqrt{2}v}} \in \cdot \bigg)
		\underset{v \to \infty}\longrightarrow \Gamma \,,
	\end{equation}
in a suitable sense. We call $\Gamma$ the {\em intensity measure of unusually large clusters}s.

For simplicity of presentation, let us first pretend that $\Gamma$ is a finite measure, so that it can be normalized into a probability measure $\wh{\Gamma}$.
Then, specifying~\eqref{e:A101.35} to $v=v_x$, we see that a cluster at $u_x$ 
contributes $v_x \rme^{\sqrt{2}v_x} \chi  = \rme^{\sqrt{2}u} \rme^{\sqrt{2}x} \chi (1+o(1))$, ``with probability'' $|\Gamma| v_x^{-1} = |\Gamma| u^{-1}(1+o(1))$, 
where $\chi$ is chosen according to $\wh{\Gamma}$.
Combined with~\eqref{e:A1.33a} we thus have
\begin{equation}
\label{e:A1.36}
\begin{split}
\frac{1}{e^{\sqrt{2} u}} \bigg[ \cE_\circ \Big([-u, \infty)\;;\; u_* + [x^-_u,\, x^+_u]\Big)\bigg] & = (1+o(1)) \int_{x_u^-}^{x_u^+} \rme^{\sqrt{2} x} \,\frac{\cC_x([-v_x,\,0])}{v_x \rme^{\sqrt{2}v_x}}\,\cE_\circ^*\big(u_* + \rmd x\big) \\
& \cong C \int_{t_u^-}^{t_u^+} t\, \chi_t \cN(\rmd t) \,,
\end{split}
\end{equation}
where $\cN \sim {\rm PPP}\big((1/\sqrt{2}) t^{-2} \rmd t\big)$, 
$t^\pm_u = \rme^{\sqrt{2} x_u^\pm}$, $(\chi_t)_{t \geq 0}$ are i.i.d. ``copies'' of $\chi$, and we have employed the change of variables $t = \rme^{\sqrt{2}x}$.

As $t^+_u$ and $t^-_u$ tend to $\infty$ and zero respectively as $u \to \infty$, a standard Laplace Transform computation shows that after subtracting a compensating term, for example, 
\begin{equation}
\label{e:A1.37}
	\bbE \int_{t_u^-}^{0} t\, \chi_t \cN(\rmd t) 
	\approx \frac{1}{e^{\sqrt{2} u}} \bbE\, \cE_\circ \Big([-u, \infty)\;;\; u_* + [x^-_u,\, 0]\Big) \,,
\end{equation}
the last integral in~\eqref{e:A1.36} tends weakly to a spectrally positive stable-$1$ random variable.
While $\Gamma$ is, in reality, infinite, we show that it obeys $\Gamma ((y, \infty)) \lesssim y^{-2} \wedge  (\log y^{-1})^{5/2}$, which provides sufficient control on its growth and decay near zero and infinity, respectively. This allows for a similar derivation as in the finite case to go through.

Lastly, we draw attention to the fact that the mean of $\cE_\circ \big([-u, \infty)\;;\; u_* + (0, x_u^+]\big)$ is absent from the centering function in~\eqref{e:A1.26}. While one can replace $0$ with any other constant in~\eqref{e:A1.26} (and then adjust $C_\circ$ accordingly), replacing $0$ by $x_u^+$ (or any other function $x_u^{+\prime}$ that grows with $u$) will necessitate a growing function in $u$ in place of the constant $C_\circ$. As in the case of~\eqref{e:A1.37}, the mean $\cE_\circ \big([-u, \infty)\;;\; u_* + (x_u^{+\prime}, x_u^+]\big)$ does not capture the typical value of this quantity, which tends to $0$ in probability with $u$.

\subsubsection{Other log-correlated fields and universality}
\ifonlystable
Let us briefly remark that another advantage of the probabilistic-constructive argument for Theorem~\ref{t:A1.5} is that, unlike the PDE approach, it applies to other log-correlated fields. Indeed, many such models are shown or conjectured to exhibit similar extreme order statistics. In particular, they should all have the same form for the limiting extremal process and a similar description for the cluster law. We thus expect that an analog of Theorem~\ref{t:A1.5} will hold for such models as well, i.e. that the Stable-$1$ fluctuations are also a universal feature of the extreme values of log-correlated fields. 
\fi

This ``prediction'' should be rather easy to verify in the cases of the discrete Gaussian free field (DGFF) on planar domains and the branching random walk (BRW) where the theory of extreme order statistics is fully developed (see, e.g.,~\cite{BiskupNotes} for the DGFF case and~\cite{Shi} for the BRW case) and the technology is sufficiently advanced (e.g., the concentric decomposition from~\cite{BisLou18} and the Ballot Theorems in~\cite{Ballot20}, for the case of the DGFF).

\paragraph{Organization of the paper}
\ifonlystable
The remainder of the paper is organized as follows. Section~\ref{s:2} includes preliminary results which are needed for the proofs to follow. These include an account of the weak representation of the cluster law from~\cite{CHL19, CHL20}, Ballot Estimates and the joint convergence of the critical extreme values and the maximum (Proposition~\ref{prop:6.5}). This section also lays down any general notation used throughout. Proofs of the main theorem and corollaries are given in Section~\ref{s:B3}, assuming estimates for the intensity measure of unusually large clusters and its limit (Theorem~\ref{stable_key_lemma}). The latter are proved in the following Section~\ref{s:A4}. Section~\ref{s:A5} includes proofs of various technical lemmas which were deferred so as not to hide the wood for the trees. The appendix includes proof of the preliminary statements.
\fi

%% file: preliminaries.tex
In this section we include preliminary statements and notation which will be used throughout this manuscript. 
All of the results here either appear in the literature, or require  proofs which are fairly standard by now. References are provided for the first kind of results. Proofs for results of the second kind, unless very short, are relegated to the appendix, so that attention can be given to the main and novel parts of the arguments in this work. 

\subsection{Additional notation}
If $f$ is a function on some space $\cX$, we shall write $f_A$ for its restriction to $A \subset \cX$. If $\mu$ is a measure on $(\cX, \Sigma)$ and $f$ is non-negative and $\Sigma$-measurable, then we write $\mu(\rmd x) f(x)$ to denote a measure on $(\cX, \Sigma)$ whose Radon-Nikodym derivative w.r.t.~$\mu$ is $f$. If $\cX$ is also metric (or topological), then we write $C(\cX)$ for the space of continuous functions on $\cX$, which we always equip with the topology of uniform convergence on compact sets. 

Throughout this manuscript, $W = (W_s)_{s \geq 0}$ will always denote a standard Brownian motion. For this process or any other Markovian process $Z = (Z_t)_{t \geq 0}$, we shall formally, but conventionally, write $\bbP(-|Z_0=x)$ to denote the case when $Z$ is defined to have $Z_0 = x$. When there is no ambiguity concerning the underlying process, we shall abbreviate:
\begin{equation}
\label{e:102.1}
	\bbP(-) \equiv \bbP(-\,|\,Z_0 = 0\big)
	\ , \quad 
	\bbP_{r,x}(-) \equiv \bbP(-\,|\,Z_{r} = x\big) 
	\  , \quad
	\bbP_{r,x}^{s,y}(-) \equiv \bbP(-\,|\,Z_{r} = x, Z_{s} = y) \,.
\end{equation}

Finally, we shall write $a \lesssim b$ if there exists $C \in (0,\infty)$ such that $a \leq C b$ and $a \asymp b$ if both $a \lesssim b$ and $b \lesssim a$. To stress that $C$ depends on some parameter $\alpha$ we shall add a subscript to the relation symbols. As usual $C$, $c$, $C'$, etc. denote positive and finite constants whose value may change from one use to another.

\subsection{Weak representation of the cluster law}
\label{s:2.1}
In this subsection we include a subset of the theory that was developed in~\cite{CHL19} in order to get a handle on the law of the clusters. The analysis of the cluster distribution is done in three steps, which correspond to the next three sub-sections. In the first step one expresses the law of the clusters as the weak limit as $t \to \infty$ of the relative heights of the genealogical neighbors of a distinguished particle, called the ``spine'', which is conditioned to be the global maximum and at a prescribed height. Here one appeals to the so-called spine decomposition method (or spine method, for short). 

The second step is to recast events involving the spine being a global maximum, as events involving a certain ``decorated random-walk-like process'' which is required to stay below a given curve. We refer to such events as ``Ballot'' events, on account of the analogy with a simple random walk which is restricted to be positive. In the third step one estimates the probability of such Ballot events. This is achieved by appealing to Ballot estimates for standard Brownian motion.

\subsubsection{Spinal decomposition}
The {\em $1$-spine branching Brownian motion} (henceforth sBBM)  evolves like BBM, except that at any given time, one of the living particles is designated as the ``spine''. What distinguishes this particle from the others, is that the spine branches at Poissonian rate $2$ and not $1$ as the remaining living particles. There is no difference in the motion, which is still that of a BM. When the spine branches, one of its two children is chosen uniformly at random to become the new spine. 

We shall keep the same notation $(h_t, T_t, \rmd_t)_{t \geq 0}$ for the genealogical 
and positional processes associated with this model. The additional information, namely the identity of the spine at any given time, will be represented by the process $(X_t)_{t \geq 0}$, where $X_t \in L_t$ for all $t \geq 0$. We shall also use the same notation as in the case of BBM to denote objects associated with the motion $h$, e.g. $\wh{h}_t$ will still denote $h_t - m_t$. To make the distinction from regular BBM explicit, we will denote the underlying probability measure by $\wt{\bbP}$ (and $\wt{\bbE}$), in accordance with the notation in~\cite{CHL19}. 

A useful representation for this process can be obtained by ``taking the point of view of the spine''. More precisely, one can think of sBBM as a process driven by a backbone Brownian motion: $(h_t(X_t) :\: t \geq 0)$, representing the trajectory of the spine, from which (non-spine) particles branch out at Poissonian rate $2$ and then evolve independently according to the regular BBM law.

The connection between sBBM and regular BBM is given by the following ``Many-To-One'' lemma. This lemma is the ``workhorse'' of the spine method, being a useful tool in moment computations for the number of particles in a branching process satisfying a prescribed condition. Recall that for $t \geq 0$, we let $\cF_t$ denote the sigma-algebra generated by $(h_s, T_s, \rmd_s)_{s \leq t}$.
\begin{lem}[Many-To-One]
\label{l:4.1}
Let $F = (F(x) :\: x \in L_t)$ be a bounded $\cF_t$-measurable real-valued random function on $L_t$. Then, 
\begin{equation}
\bbE \Big( \sum_{x \in L_t} F(x) \Big)
= \rme^{t} \,\wt{\bbE} F(X_t) \,.
\end{equation}
\end{lem}

A direct application of the above lemma, together with the statistical structure of extreme values in the limit, gives the following expression for the law of the clusters. 
\begin{lem}[Lemma~5.1 in~\cite{CHL19}]
\label{l:7.0}
For all $r > 0$, there exists a probability measure $\nu_r$ on the space $\bbM_p((-\infty, 0])$ such that for all $u \in \bbR$, 
\begin{equation}
\label{e:387}
\wt{\bbP} \big( \cC_{t, r} (X_t) \in \cdot \, \big|\, \wh{h}_t(X_t) = \wh{h}_t^* = u\big) 
\underset{t \to \infty}\Longrightarrow \nu_r
\end{equation}
and
\begin{equation}
	\nu_r \underset{r \to \infty}\Longrightarrow \nu \,,
\end{equation}
where $\cC_{t,r}$ is as in~\eqref{e:5B} and $\nu$ is as in~\eqref{e:N7}.
Both limits hold in the sense of weak convergence on $\bbM((-\infty,0])$ equipped with the vague topology. Lastly~\eqref{e:387} still holds with $\nu_r$ replaced by $\nu$, if $r$ is taken to $\infty$ together with $t$ as long as $t-r \to \infty$ as well.
\end{lem}
We note that, in the proof in~\cite{CHL19}, the limits in $r$ and $t$ were taken at the same time, but the proof carries over to case where the limits are taken sequentially (in fact, the proof in~\cite{CHL19} relies on Theorem~2.3 in~\cite{ABBS2013}, in which the limit is taken sequentially).

\subsubsection{Decorated random walk representation}
\label{s:2.1.2}
Recall that $W=(W_s :\: s \geq 0)$ denotes a standard Brownian motion. For $0 \leq s \leq t$, set
\begin{equation}
\label{e:20.5}
\gamma_{t,s} := \tfrac{3}{2\sqrt{2}} \big(\log^+ s - \tfrac{s}{t}\log^+ t \big) 
\end{equation}
and let
\begin{equation}
\wh{W}_{t,s} := W_s - \gamma_{t,s} \,.
\end{equation}
We also let $H := \big(h^s := (h^s_t)_{t \geq 0} :\: s \geq 0\big)$ be a collection of independent copies of a regular BBM process $h$, that we will assume to be independent of $W$ as well. Finally, let $\cN$ be a Poisson point process with intensity $2 \rmd x$ on $\bbR_+$, independent of $H$ and $W$ and denote by $\sigma_1 < \sigma_2 < \dots$ its ordered atoms. The triplet 
\begin{equation}
\label{e:102.27}
(\wh{W}, \cN, H)	
\end{equation}
will be referred to as a {\em decorated random-walk-like process} (DRW). 

In what follows, we shall add a superscript $s$ to any notation used for objects which were previously defined in terms of $h$, to denote objects which are defined in the same way, but with $h^s$ in place of $h$. In this way, $\wh{h}^s_t$, $\wh{h}^{s*}_t$ and $\cE_t^s$ are defined as in the introduction only with respect to $h^s$ in place of $h$.
Using this notation, for $0 \leq r < R \leq t$ and $u \in \bbR$ we define
\begin{equation}
\label{e:2.5}
\cA_t(u; r,R) := \Big \{
\max\limits_{k : \, \sigma_k \in [r,R]} \big(\wh{W}_{t,{\sigma_k}} + \wh{h}^{\sigma_k*}_{\sigma_k}\big) \leq u \Big\} \,,
\end{equation}
with the abbreviations
\begin{equation}
\label{e:102.7}
\cA_t(r,R) := \cA_t(0; r,R) 
\quad , \qquad 
	\cA_t \equiv \cA_t(0,t)\,.
\end{equation}
This is the event that the Brownian motion $W$ plus a certain functional of its decorations $H$, stay below the curve $\gamma_{t,\cdot}$ at random times which correspond to the atoms of $\cN$. We shall refer to such event as a DRW {\em Ballot} event.

The connection between probabilities of the form~\eqref{e:387} and DRW Ballot probabilities are given by the following two lemmas which were proved in~\cite{CHL19}. 
\begin{lem}[Lemma~3.1 in~\cite{CHL19}]
\label{l:2.3}
For all $0 \leq r \leq t$ and $u, w \in \bbR$,
\begin{equation}
\label{e:29.2}
\wt{\bbP} \Big( \wh{h}_t^*\big(\rmB^\rmc_{r}(X_t)\big)
 \leq u \, \Big| \, \wh{h}_t(X_t) = w \Big) 
= \bbP \big( \cA_t(r,t) \, \big|\,
	\wh{W}_{t,r} = w - u,\, \wh{W}_{t,t} = -u \big) \,.
\end{equation}
In particular for all $t \geq 0$ and $v,w \in \bbR$,
\begin{equation}
\label{e:29.1}
\wt{\bbP} \big( \wh{h}^*_t \leq u \, \big| \, \wh{h}_t(X_t) = w \big) 
= \bbP \big( \cA_t \, \big|\, 
	\wh{W}_{t,0} = w - u,\, \wh{W}_{t,t} = -u \big) \,.
\end{equation}
\end{lem}

\begin{lem}[Lemma~3.2 in~\cite{CHL19}]
\label{l:5.2}
For all $0 \leq r \leq t$,
\begin{equation}
\label{e:26.6}
\wt{\bbP}  \Big( \cC_{t,r} (X_t) 
 \in \cdot \,\Big|\, \wh{h}^*_t = \wh{h}_t(X_t) = 0 \Big)  
= \bbP \bigg( \int_0^r \! \cE_{s}^{s} \big(\cdot - \wh{W}_{t,s} \big)\cN(\rmd s) \in \,  \cdot 
	 \ \Big| \ \wh{W}_{t,0} \! = \! \wh{W}_{t,t} \! = \! 0 \,;\, \cA_t \, \bigg).
\end{equation}
\end{lem}

The following lemma will be useful when estimating increments of $\wh{W}$.
\begin{lem}[Lemma~3.3 in~\cite{CHL19}] 
\label{l:lisa}
Let $s,r, t\in \bbR$ be such that $0 \leq r \leq r+s \leq t$, then
\begin{equation}\label{e:aser1}
 - 1  \, \leq \log^+ (r+s) - \left(\tfrac{t-(r+s)}{t-r} \log^+r + \tfrac{s}{t-r} \log^+t\right)
\leq \, 1+ \log^+ \big(s \wedge (t-r-s) \big).
\end{equation}
\end{lem}

\subsubsection{Ballot estimates for DRW}
In this subsection we include estimates for Ballot probabilities involving the DRW from the previous sub-section. They are either derived from~\cite{CHL17_supplement} in which more general DRW processes are treated, or taken directly from~\cite{CHL19}, in which the results from~\cite{CHL17_supplement} were specialized to the DRW from Subsection~\ref{s:2.1.2}. The proofs of the statements here are deferred to the appendix.

The first lemma provides bounds and asymptotics for the DRW Ballot event.
\begin{lem}[Essentially Lemma~3.4 in~\cite{CHL19}]
\label{lem:15}
There exists $C, C' > 0$ such that for all $0 \leq r < R\leq t$ and $w, v \in \bbR$,
\begin{equation}
\label{e:52}
\bbP \big( \cA_t(r,R)
	\,\big|\, \wh{W}_{t,r} = v \,,\,\, \wh{W}_{t,R} = w \big)
	\leq C \frac{(v^- + 1)(w^- + 1)}{R-r} \,.
\end{equation}
Also, there exists non-increasing functions $g: \bbR \to (0,\infty)$ and $f^{(r)}: \bbR \to (0,\infty)$ for $r \geq 0$, such that for all such $r$
\begin{equation}
\label{e:53}
\bbP \big( \cA_t(r,R)
	\,\big|\, \wh{W}_{t,r} = v \,,\,\, \wh{W}_{t,R} = w \big)  
	\sim 2 \frac{f^{(r)}(v) g(w)}{R-r},
\end{equation}
as $R \to \infty$ uniformly in $t \geq R$ and $v,w$ satisfying $v,w < 1/\epsilon$ and $(v^-+1)(w^-+1) \leq R^{1-\epsilon}$ for any fixed $\epsilon > 0$. Moreover, 
\begin{equation}
\label{e:53.5}
\lim_{v \to \infty} \frac{f^{(r)}(-v)}{v} = \lim_{w \to \infty} \frac{g(-w)}{w} = 1  \,,
\end{equation}
for any $r \geq 0$. Finally there exists $f: \bbR \to (0, \infty)$ such that for all $v \in \bbR$, 
\begin{equation}
\label{e:54}
f^{(r)}(v) \overset{r \to \infty} \longrightarrow f(v) \,.
\end{equation}
\end{lem}

We shall also need the following version of the previous lemma,  where the ballot barrier is at different heights in both ``edges'' of the time interval.
\begin{lem}\label{lemma_mixed_ballot} 
Let $r: \bbR_+ \to \bbR$ be such that $r_s = o(s), r_s\to \infty$ as $s \to \infty$ and $r_s, s-r_s$ are both uniquely determined by $s$. Then, 
	\begin{equation}
		\bbP \big(\cA_t(w + z; r, s) \cap \cA_t(0, r)
		\,\big|\, \wh{W}_{t,0} = 0 \,,\,\, \wh{W}_{t,s} = w \big) \sim  \frac{2f^{(0)}(0) g(-z)}{s}, \label{eq:ballot_prob_0_to_s}
	\end{equation}
	as $s\to \infty$, uniformly for $t > s$ and $w, z$ in compact sets. Also,
	\begin{equation}
		\bbP \big(\cA_t(w + z; s, t/2) \cap \cA_t(t/2, t) 
				\,\big|\, \wh{W}_{t,s} = w \,,\,\, \wh{W}_{t,t} = 0 \big)
		\sim  \frac{2f(-z) g(0)}{t}, 
		\label{eq:ballot_prob_s_to_t}
	\end{equation}
	as $t \to \infty$ followed by $s\to \infty$, uniformly for $w, z$ in compact sets.
\end{lem}

The next lemma shows that under the Ballot event, the process $\wh{W}_{t,\cdot}$ is (entropically) repelled away from $0$.
\begin{lem}
\label{l:10.7.0}
For all $M > 0$, $2 \leq s \leq t/2$,
\begin{equation}
\bbP \Big( \max_{u\in[s,\, t/2]}
	\wh{W}_{t,u} > -M \,\Big|\, \wh{W}_{t,0} = 0 \,,\,\, \wh{W}_{t,t} = 0\,,\,\,\cA_t\Big)
	\leq C \frac{(M+1)^2}{\sqrt{s}} \,.
\end{equation}
\end{lem}

The last lemma shows that one can replace the restriction on $\wh{W}_{t,u}$ for $u \in [s,t-s]$ under $\cA_t$ with the condition that $W_{u}$ (without the additional $\gamma_{t,s}$) stays negative within this interval, with the difference between the probabilities of the two events being of smaller order.
\begin{lem}
\label{l:2.13}
For $0 < r < R < t$, denote the event
\begin{equation} \wt{\cA}_{t}(r, R) := \{\cA_t(0,r) \cap \cA_t(R,t) \cap \{\max W_{[r,R]} \leq 0\} \}.
\end{equation}
There exists $C > 0$ such that for all $2 < s < t/2$ \,,
\begin{equation}
    \bbP \Big(\cA_{t} \, \Delta \,  \wt{\cA}_{t}(s, t-s)\,\Big|\, \wh{W}_{t,0} = 0 \,,\,\, \wh{W}_{t,t} = 0 \Big) \leq C s^{-1/4} t^{-1} \,.
\end{equation}
\end{lem}

\subsection{Joint convergence of critical extreme level-sets and maximum}
Lastly, we shall also need the following proposition, which gives the joint limit in law of the centered maximum and the scaled size of the critical extreme (super) level sets. The latter are defined as the set of particles at time $t$ whose height is (at least) $m_t - \Theta(\sqrt{t})$. While such convergence is well-known, we could not find a proof for this in existing literature and thus included a short and standard proof in Appendix~\ref{s:A3}. We remark that a convergence result for the super level sets up-to $o(t/\log t)$ below $m_t$ were very recently treated in~\cite{Gabriel25} (albeit not jointly with the centered maximum).

\begin{prop}\label{prop:6.5} 
There exists $c_0 \in (0,\infty)$, such that uniformly in $v \in t^{1/2} [\delta, \delta^{-1}]$ for any $\delta > 0$, 
	\begin{equation}
	\label{e:6.47}
		\left(\frac{\cE_t\big([-v, \infty)\big)}{v e^{\sqrt{2}v - \frac{v^2}{2t}} }, \, \wh{h}^*_{t}\right) \underset{t \to\infty}\Longrightarrow \big(c_0 Z,\,  \wh{h}_\infty^* \big)
	\quad ; \qquad G + \tfrac{1}{\sqrt{2}} \log Z \,,
	\end{equation}
where $Z$ and $G$ are as in~\eqref{e:101.4}.
\end{prop}

%% file: stable.tex
\ifonlystable
\section{Stable limit of compensated process}
\label{s:B3}
\fi

Letting $x^\pm$ be as in the theorem, the target quantity we wish to study is
\begin{equation}
	\wt{\mathcal{R}}_u := e^{-\sqrt{2} u} \cE_\circ\left([-u, \infty); -\tfrac{1}{\sqrt{2}} \log u + [x^-_u, x^+_u]\right) \,.
\end{equation}
Denoting
\begin{equation}\label{def:Y_u_x}
	Y_u(x) := e^{-\sqrt{2} (u+x)} \cC^{(x)}([-w_u(x), 0])
	\quad , \qquad w_u(x) := u - \frac{1}{\sqrt{2}} \log u + x \,, 
\end{equation}
we may write $\wt{\mathcal{R}}_u$ as
\begin{equation}
	\wt{\mathcal{R}}_u = \int_{x_u^-}^{x^+_u} e^{-\sqrt{2} u} \cC^{(x)}([-w_u(x), 0]) \, \rmd \cE_\circ^*\big(x -\tfrac{1}{\sqrt{2}} \log u \big) 
	= \int_{x_u^-}^{x^+_u} e^{\sqrt{2} x} \, Y_u(x) \, \cE^*_u(dx) \,,
\end{equation}
where 
$\cE^*_u(\cdot) := \cE_\circ^*(\cdot - \tfrac{1}{\sqrt{2}} \log u)$ is a PPP on $\bbR$ with intensity $u \rme^{-\sqrt{2}x} \rmd x$.

Using the above representation, the Laplace Transform of 
$\wt{\mathcal{R}}_u$ is $\bbE \exp(-\lambda \wt{\mathcal{R}}_u) = \rme^{-\wt{\phi}_u(\lambda)}$, where
\begin{equation}
\begin{split}
\label{hat_phi_u}
	\wt{\phi}_u(\lambda) &:= \int_{x^-_u}^{x^+_u} u e^{\sqrt{2}x} dx \int_0^\infty \left(1-e^{-\lambda e^{\sqrt{2}x} y} \right) \bbP(Y_u(x)\in dy) \\
	&= \tfrac{1}{\sqrt{2}} \int_{t^-_u}^{t^+_u} t^{-2} dt \int_0^\infty \left(1-e^{-\lambda t y} \right) u \bbP(Y_u(x(t))\in dy) \,,
\end{split}
\end{equation}
where we performed the change of variable $x(t) = \tfrac{1}{\sqrt{2}} \log t$ and set
\begin{equation}
	t^{\pm}_u := e^{\sqrt{2} x^\pm_u} \,.
\end{equation} 

For the latter to converge to a finite limit, we must subtract a {\em deterministic} compensating term. To this end, we pick an arbitrary $\rho > 0$ and set
\begin{equation}
\label{e:A5.7}
	E_u(\rho) := \tfrac{1}{\sqrt{2}} \int_{t^-_u}^{t^+_u} t^{-2} dt \int_0^\infty ty \mathbbm{1}_{ty\leq \rho} \ u \bbP(Y_u(x(t))\in dy) \,.
\end{equation}
Then, the Laplace Transform of the \emph{compensated} variable 
\begin{equation}
\mathcal{R}_u := \wt{\mathcal{R}}_u - E_u(\rho) \,,
\end{equation}
is $\bbE \exp\left(-\lambda \mathcal{R}_u\right) = \exp(-\phi_u(\lambda))$,
where
\begin{equation}\label{eq:def_phi_u}
	\phi_u(\lambda) :=  \wt{\phi}_u(\lambda)-\lambda E_u(\rho) = \tfrac{1}{\sqrt{2}} \int_{t^-_u}^{t^+_u} t^{-2} dt \int_0^\infty \left(1-e^{-\lambda t y} - \lambda ty \mathbbm{1}_{ty\leq \rho} \right) u \bbP(Y_u(x(t))\in dy).
\end{equation}

The goal now is to show the limit of $\phi_u(\lambda)$ as $u\to \infty$ exists and has the appropriate form. For this purpose, we need to understand the limiting behavior of the (unbounded) measures $u \bbP(Y_u(x(t))\in dy)$ on $(0, \infty)$, in the regime $t \in [t^-_u, t^+_u]$. Equivalently, we may study the $w \to \infty$ limit of
\begin{equation}
\label{e:5.10}
w \bbP(X(w) \in \rmd x) \,,
\end{equation}
where 
\begin{equation}
	X(w) := \frac{\cC([-w, 0])}{w e^{\sqrt{2}w}} \,.
\end{equation}
Indeed, for any $y > 0$, as $u \to \infty$,
\begin{equation}\label{eq:rel_Y_u_X_w}
	Y_u(x) \overset{d}{=} X(w_u(x)) (1+o(1))
\end{equation}
and
\begin{equation}\label{eq:rel_Y_u_X_w_measure}
	u \bbP(Y_u(x(t)) > y) = w \bbP \Big(X(w) > y(1+o(1))\Big)\, \big(1+o(1)\big)\,,
\end{equation}
where $w = w_u(x(t))$ with the $o(1)$ terms uniform in $t \in [t^-_u, t^+_u]$ and independent of $y$. 

\ifonlystable
\subsection{The intensity of unusually large clusters}
\fi
\label{s:5.1.1}
We will show that the measures in~\eqref{e:5.10} indeed admit a weak-type limit, namely that
\begin{thm}\label{stable_key_lemma}
	There exists an (infinite) measure $\Gamma$ on $(0, \infty)$ such that for every continuity point $y$ of $\Gamma$, 
	\begin{equation}\label{convergence_conjecture}
		w \bbP(X(w) > y) \underset{w \to \infty} \longrightarrow \Gamma ((y, \infty)) 
	\end{equation}
Moreover, 
\begin{equation}\label{eq:limiting_measure_1st_moment}
		\int_0^\infty y\, \Gamma(dy) = C_\star,
\end{equation}
where $C_\star$ is as in~\eqref{e:1.10}. Finally, for all $y > 0$, 
\begin{equation}
\label{e:A5.15}
	\Gamma ((y, \infty)) \lesssim y^{-2} \wedge  (\log y^{-1})^{5/2} \,.
\end{equation}
\end{thm}
The proof of this theorem is given in \ifonlystable
Section \ref{s:A4}
\fi and constitutes the main effort of proving Theorem \ref{t:A1.5}. 

We shall also need the following two lemmas, which gives estimates and bounds on the moments and tails of the pre-limit $w \bbP(X(w) \in \rmd x)$. The first is essentially a rewriting of \cite[Proposition 1.5]{CHL19}.
\begin{lem}
\label{l:A5.2}
As $w\to \infty$, with $C_\star$ as in~\eqref{e:1.10},
\begin{equation}\label{estimate_0_X_0_1}
	w \bbE X(w) \sim C_\star,
\end{equation} 
and
\begin{equation}\label{estimate_0_X_0_2}
	w \bbE X(w)^2 \leq C (1+ w^{-1}).
\end{equation}
In particular, 
\begin{equation}\label{estimate_0_X_2}
	w \bbP \left(X(w) > y  \right) \lesssim y^{-1} \wedge y^{-2}.
\end{equation}
and for all $w\geq 1$,
\begin{equation}\label{estimate_0_X_1}
	w \bbE \left(X(w); X(w) > y  \right) \lesssim y^{-1}.
\end{equation}
\end{lem}
The second is an improved version of~\cite[Proposition 1.1]{CHL20}.
\begin{lem}\label{improved_estimates_X} For all $w \geq 0$ and $0 < y < 0.9$,
	\begin{equation}\label{estimate_1_X}
		w \bbE \left(X(w); X(w) \leq y \right) \lesssim we^{-w/2} + y (\log y^{-1})^{5/2}
	\end{equation}
	and	
	\begin{equation}\label{estimate_2_X}
		w \bbP \left(X(w) > y  \right) \lesssim y^{-1} w e^{-w/2} +  (\log y^{-1})^{5/2}.
	\end{equation}
\end{lem}
The proof of this lemma is just an application of 
\ifonlystable
Lemma~3.6 of~\cite{hartung2024growth},
\fi 
and the details are deferred to to
\ifonlystable
Section~\ref{s:A5}.
\fi

\ifonlystable
\subsection{Proof of Theorem~\ref{t:A1.5} and its corollaries} 
\fi
Assuming the results of Subsection~\ref{s:5.1.1}, we can now conclude the proof of Theorem~\ref{t:A1.5} and the resulting corollaries. The proof constitutes three steps, phrased below as three lemmas. The first of the three, allows us to pass to the $u \to \infty$ limit in~\eqref{eq:def_phi_u}.
\begin{lem}\label{stable_main_convergence} 
As $u\to \infty$, the integral 
\begin{equation}
	\phi_u(\lambda) = \tfrac{1}{\sqrt{2}} \int_{t^-_u}^{t^+_u} t^{-2} dt \int_0^\infty \left(1-e^{-\lambda t y} - \lambda ty \mathbbm{1}_{ty\leq \rho} \right) u \bbP(Y_u(x(t))\in dy)
\end{equation}
	converges to
	\begin{equation}
	\label{e:A5.21}
	 \phi(\lambda) := \tfrac{1}{\sqrt{2}} \int_{0}^{\infty} t^{-2} dt \int_0^\infty \left(1-e^{-\lambda t y} - \lambda ty \mathbbm{1}_{ty\leq \rho} \right) \Gamma (dy).
	\end{equation}
\end{lem}

Next, in order to replace $E_u(\rho)$ by the normalized mean in~\eqref{e:A1.26}, we have
\begin{lem}
\label{l:5.5}
For each $\rho > 0$, there exists $C_\rho \in (-\infty, \infty)$ such that
\begin{equation}
E_u(\rho) - \rme^{-\sqrt{2}u} \bbE \cE_\circ\big([-u, \infty); -\tfrac{1}{\sqrt{2}} \log u + [x^-_u, 0]\big) \longrightarrow C_\rho \,.
\end{equation}
\end{lem}

Lastly, we shall also need to control the order of fluctuations of the contributions from clusters whose tip is sufficiently far from $-\tfrac{1}{\sqrt{2}} \log u$.
\begin{lem}
\label{l:A5.6}
For any $\epsilon > 0$,
\begin{equation}
e^{-\sqrt{2} u} \cE_\circ\Big([-u, \infty); \big(\big(-\tfrac{1}{\sqrt{2}}+\epsilon\big) \log u, \infty\big)\Big)
\underset{u \to \infty}{\overset{\text{a.s.}}\longrightarrow} 0 \,.
\end{equation}
and
\begin{equation}
e^{-\sqrt{2} u} \Big[\cE_\circ\Big([-u, \infty); \big((-\infty, -\tfrac{1}{\sqrt{2}}-\epsilon) \log u\big)\Big)  - \bbE \cE_\circ\Big([-u, \infty); \Big(\big(-\infty, -\tfrac{1}{\sqrt{2}}-\epsilon\big) \log u\Big)\Big)\Big]
\underset{u \to \infty}{\overset{\bbP}\longrightarrow} 0 \,.
\end{equation}
\end{lem}

Deferring momentarily the proof of the above lemmas (which will be given in Subsection \ref{s:5.1.3}), let us give
\begin{proof}[Proof of Theorem~\ref{t:A1.5}]
Thanks to Lemma~\ref{stable_main_convergence} and the derivation in the beginning of this section, we have
\begin{equation}
\cR_u \underset{u \to \infty} \Longrightarrow \wt{\cR}\,,
\end{equation}
where $\wt{\cR}$ is a random variable with Laplace Transform $\bbE \rme^{-\lambda \wt{\cR}} = \exp(-\phi(\lambda))$, with $\phi$ as in~\eqref{e:A5.21}. The latter can be written as
\begin{equation} \phi(\lambda) = \tfrac{1}{\sqrt{2}} \int_{0}^{\infty} \left(1-e^{-\lambda z} - \lambda z \mathbbm{1}_{z\leq \rho} \right) \wt{\Gamma} (dz)
\end{equation}
where $\wt{\Gamma}$ is the (improper) distribution of $z = ty$ when $t$ has distribution $t^{-2} dt$ and $y$ has distribution $\Gamma(dy)$. Using Fubini, it is easy to compute that for $a > 0$, 
\begin{align}\label{computing_Gamma}
	\wt{\Gamma}(z > a) = \int_0^\infty \Gamma(dy) \int_0^\infty \mathbbm{1}_{ty > a} \ t^{-2}   dt
	= \int_0^\infty  \frac{y}{a}\, \Gamma(dy)
	= \frac{C_\star}{a} \,,
\end{align}
with the last equality following by~\eqref{eq:limiting_measure_1st_moment}.
Thus, $\wt{\Gamma}(dz) = C_\star z^{-2} dz$, and so we have
\begin{equation} 
\phi(\lambda) =  \tfrac{C_\star}{\sqrt{2}} \int_{0}^{\infty} \frac{1-e^{-\lambda z} - \lambda z \mathbbm{1}_{z\leq \rho}}{z^2} \, dz 
= -\tfrac{C_\star}{\sqrt{2}}  \left(\lambda \log \lambda - (\log \rho - 1 + \gamma) \lambda \right) \,,
\end{equation}
where $\gamma$ is the Euler-Mascheroni constant. This may be obtained, for example, by differentiating under integral twice to get the explicit form of $\phi''(\lambda)$. Setting
\begin{equation}
	C_3 := \tfrac{C_\star}{\sqrt{2}}(\log \rho - \log \tfrac{C_\star}{\sqrt{2}} + \gamma) + C_\rho \,,
\end{equation}
where $C_\rho$ is as in Lemma~\ref{l:5.5}, the Laplace Transform of the  left hand side in~\eqref{e:A1.26} thus tends to the Laplace Transform of the right hand side thereof.

For the second part of the theorem, pick any $\epsilon > 0$ and set 
$x_u^{-'} = x_u^- \wedge (-\epsilon \log u)$ and 
$x_u^{+'} = x_u^+ \vee \epsilon \log u$.
Then write the difference on left hand side of~\eqref{e:A1.27} as
\begin{equation}
\begin{split}
\Big(\cE_\circ \big([-u, \infty); & -\tfrac{1}{\sqrt{2}}\log u +(-\infty, x_u^{-'}]\big) - 
	\bbE \cE_\circ \big([-u, \infty); -\tfrac{1}{\sqrt{2}}\log u +(-\infty, x_u^{-'}]\big)\Big) \\
& + \Big(\cE_\circ \big([-u, \infty); -\tfrac{1}{\sqrt{2}}\log u +[x_u^{-'},\, x_u^{+'}]\big)  
- \bbE \cE_\circ \big([-u, \infty); -\tfrac{1}{\sqrt{2}}\log u +[x_u^{-'},\, 0]\big)\Big) \\
& - \Big(\cE_\circ \big([-u, \infty); -\tfrac{1}{\sqrt{2}}\log u +[x_u^-,\, x_u^+]\big) 
- \bbE \cE_\circ \big([-u, \infty); -\tfrac{1}{\sqrt{2}}\log u +[x_u^-,\, 0]\big)\Big) \\
& + \cE_\circ \big([-u, \infty); -\tfrac{1}{\sqrt{2}}\log u + (x_n^{+'}, \infty)\big) \,.
\end{split}
\end{equation}
After normalization by $\rme^{-\sqrt{2} u}$ the terms in the first and last lines above tend to $0$ in probability as $u \to \infty$, thanks to Lemma~\ref{l:A5.6}. 
As for the difference of the middle two lines, thanks to Lemma~\ref{l:5.5}, we can replace the normalized means with the quantities $E'_u(\rho)$, $E_u(\rho)$, respectively, with the former defined as in~\eqref{e:A5.7} only with respect to $t_u^{\pm'} := \rme^{\sqrt{2}x_u^{\pm'}}$ in place of $t_u^\pm$. Then the Laplace Transform of the normalized difference is equal to $\rme^{-(\phi'_u(\lambda) - \phi_u(\lambda))}$, where $\phi'_u(\lambda)$ is again defined as in~\eqref{eq:def_phi_u} only w.r.t. $t_u^{\pm'}$. But then, by Lemma~\ref{stable_main_convergence} and the finiteness of the integral in~\eqref{e:A5.21}, the difference $\phi_u(\lambda) - \phi'_u(\lambda)$ tends to $0$ with $u$. This implies that also the normalized difference of the middle two lines converges in probability to $0$ as $u \to \infty$, which combined with the same convergence of the other terms shows the desired second part of Theorem~\ref{t:A1.5}.
\end{proof}

\begin{proof}[Proof of Corollary~\ref{c:A1.6}]
Immediate from the two statements of Theorem~\ref{t:A1.5}.
\end{proof}

\begin{proof}[Proof of Corollary~\ref{c:A1.7}]
Abbreviating $\tilde{u} := u + \frac{1}{\sqrt{2}} \log Z$, in view of the relation between $\cE$ and $\cE_\circ$, the left hand side in~\eqref{e:A1.30} is equal to
\begin{multline}
\label{e:A1.34}
\frac{Z}{\rme^{\sqrt{2} \tilde{u}}}
\bigg[
	\Big(\cE_\circ \big([-\tilde u,\, \infty)\big) - \bbE\, \cE_\circ \Big([-\tilde u,\, \infty)\;;\; \big[-\tilde u,\, -\tfrac{1}{\sqrt{2}} \log \tilde{u} \big)\Big) - C_\circ \rme^{\sqrt{2}\tilde u} \\
	- 
	\bbE \cE_\circ \Big([-\tilde u,\, \infty)\;;\; \big[-\tfrac{1}{\sqrt{2}} \log \tilde{u} ,\, -\tfrac{1}{\sqrt{2}} (\log u + \log Z)\big)\Big) \bigg] \,,
\end{multline}
where the last mean is to be subtracted instead of added, if the right point of the last interval is smaller than its left point.

Since $\log \wt{u} = \log u + o(1)$ the second mean above tends to $(\tfrac{C_\star}{\sqrt{2}} + o(1)) \rme^{\sqrt{2}\tilde{u}}\log Z$ by
\ifonlystable
Lemma~3.1 in~\cite{hartung2024growth},
\fi
where both $o(1)$ tend to $0$ as $u \to \infty$. Using that $\wt{u} \to \infty$ as $u \to \infty$, Corollary~\ref{c:A1.6} then implies that the quantity in~\eqref{e:A1.34} tends weakly to
\begin{equation}
	\frac{C_\star}{\sqrt{2}} \big(Z\cR_1 - Z \log Z\big)  \,,
\end{equation}
which has precisely the desired law.
\end{proof}

\ifonlystable
\subsection{Proof of Lemmas~\ref{stable_main_convergence}-\ref{l:A5.6}}
\fi
\label{s:5.1.3}

First, we give:
\begin{proof}[Proof of Lemma~\ref{l:A5.6}] The first statement is immediate from 
\ifonlystable
Proposition~3.2 in~\cite{hartung2024growth}
\fi
(with $\delta = \tfrac{1}{\sqrt{2}}-\eps$). The second statement follows from a direct application of the Chebyshev Inequality using the variance estimate 
\ifonlystable
in Lemma~3.1 of~\cite{hartung2024growth}.
\fi
\end{proof}

Turning to the proof of Lemma~\ref{stable_main_convergence}, we first separate the inner integral in \eqref{eq:def_phi_u} into two parts:
\begin{align}
	\Phi^{\leq\rho}_u(t) &:= \int_0^\infty (1- e^{-\lambda ty} - \lambda ty ) \mathbbm{1}_{ty\leq \rho} \cdot u\bbP(Y_u(x(t)) \in dy) \label{eq:def_phi_u_<_rho} \\
	\Phi^{>\rho}_u(t) &:= \int_0^\infty (1- e^{-\lambda ty}) \mathbbm{1}_{ty > \rho} \cdot u\bbP(Y_u(x(t)) \in dy) \label{eq:def_phi_u_>_rho}
\end{align}
Then, Lemma~\ref{stable_main_convergence} follows immediately from the next two results and the Dominated Convergence Theorem.
\begin{lem}\label{lem:convergence_inner_integral} If $\tfrac{\rho}{t}$ is a continuity point of $\Gamma$, then
	\begin{align}\label{eq:convergence_inner_integral_>_rho}
		\Phi^{>\rho}_u(t) &\underset{u\to\infty}\longrightarrow \Phi^{>\rho}(t) := \int_0^\infty (1- e^{-\lambda ty}) \mathbbm{1}_{ty > \rho}\, \Gamma(dy)  \\
		\Phi^{\leq\rho}_u(t) &\underset{u\to\infty}\longrightarrow \Phi^{\leq\rho}(t) := \int_0^\infty (1- e^{-\lambda ty} - \lambda ty ) \mathbbm{1}_{ty\leq \rho}\, \Gamma(dy). \label{eq:convergence_inner_integral_leq_rho}
	\end{align}
\end{lem}

\begin{lem}\label{lem:bound_inner_integral} There exists a finite constant $C_\rho$ such that
	\[
	\Phi^{>\rho}_u(t), \Phi^{\leq\rho}_u(t) \leq C_\rho\, \min(t^2, |\log t|^{5/2}).
	\]
	for $t\in [t^-_u, t^+_u]$, for all $u$ large enough. 
\end{lem}

\begin{proof}[Proof of Lemma \ref{lem:convergence_inner_integral}]  
By the asymptotic equivalence \eqref{eq:rel_Y_u_X_w_measure}, the convergence \eqref{convergence_conjecture} implies that at every continuity point $y$ of $\Gamma$,  
\begin{equation}\label{convergence_conjecture_applied}
	u \bbP(Y_u(x(t)) > y) \underset{u \to \infty}  \longrightarrow \Gamma ((y, \infty))
\end{equation}
Since $\Gamma$ has finite mass away from $0$,  \eqref{convergence_conjecture_applied} implies weak convergence of the (finite) measures restricted to $(y, \infty)$, whenever $y > 0$ is a continuity point of $\Gamma$. In particular, this means if $\tfrac{\rho}{t}$ is a continuity point of $\Gamma$, then 
	\[
	\Phi^{>\rho}_u(t) \to \Phi^{>\rho}(t)
	\]
	and, for $0 < \eps < \rho$ such that $\tfrac{\eps}{t} > 0$ is also a continuity point of $\Gamma$,
	\[
	\Phi^{\leq\rho}_u(t) - \Phi^{\leq\eps}_u(t) \to \Phi^{\leq\rho}(t) - \Phi^{\leq\eps}(t).
	\]
	This implies $\Phi^{\leq\rho}_u(t) \to \Phi^{\leq\rho}(t)$, if we can show that $\Phi^{\leq\eps}_u(t), \Phi^{\leq\eps}(t) \to 0$ as $\eps \to 0$ uniformly for $u$ large enough. Indeed, by Taylor expansion, the asymptotic equivalence \eqref{eq:rel_Y_u_X_w_measure}, and the 1st moment estimate \eqref{estimate_0_X_0_1}
	\begin{align}
		|\Phi^{\leq\eps}_u(t)| &\lesssim \int_0^{\tfrac{\eps}{t}} (ty)^2 \cdot u\bbP(Y_u(x(t)) \in dy) \\
		&\leq \eps t  \int y \cdot u\bbP(Y_u(x(t)) \in dy)  \\
		&= \eps t \cdot u \bbE Y_u(x(t)) \sim \eps t \cdot w \bbE X(w) \sim C_\star \eps t 
	\end{align}
	For the same reason, $|\Phi^{\leq\eps}(t)| \lesssim \eps t$. 
\end{proof}

\begin{proof}[Proof of Lemma \ref{lem:bound_inner_integral}] Applying Taylor expansion and the asymptotic equivalence \eqref{eq:rel_Y_u_X_w},
	\begin{align}
		t^{-2}|\Phi^{\leq\rho}_u(t)| &\lesssim \int_0^{\tfrac{\rho}{t}} y^2 \cdot u\bbP(Y_u(x(t)) \in dy) \\
		&= u\bbE \big[Y_u(x(t))^2 \mathbbm{1}_{Y_u(x(t)) \leq \tfrac{\rho}{t}} \big] \\
		&\sim  w\bbE \big[X(w)^2 \mathbbm{1}_{X(w) \leq  (1+o(1))\tfrac{\rho}{t}} \big] \leq C (1 + w^{-1})
	\end{align}
	where in the last step we simply dropped the truncation and apply the 2nd moment estimate \eqref{estimate_0_X_0_2}. If $\tfrac{\rho}{t} < 0.9$, then we can instead apply estimate \eqref{estimate_1_X} from Lemma \ref{improved_estimates_X} to get
	\begin{align}
		t^{-2}|\Phi^{\leq\rho}_u(t)| &\lesssim \tfrac{\rho}{t}\cdot  w\bbE \big[X(w) \mathbbm{1}_{X(w) \leq  (1+o(1))\tfrac{\rho}{t}} \big] \\
		&\lesssim \tfrac{\rho}{t}\cdot \big( we^{-w/2} + \tfrac{\rho}{t} \big(\log \tfrac{t}{\rho}\big)^{\frac{5}{2}}\big) \lesssim_\rho   t^{-2} (\log t)^{\frac{5}{2}}
	\end{align}
	where in the last step we used that $x^\pm_u$ are sub-linear in $u$ and $w\sim u$. 
	
	Similarly, by the asymptotic equivalence \eqref{eq:rel_Y_u_X_w_measure} and the tail estimate \eqref{estimate_0_X_2},
	\begin{align}
		\Phi^{>\rho}_u(t) &\leq u\bbP(Y_u(x(t)) > \tfrac{\rho}{t}) \\
		&\sim w \bbP(X(w) > (1+o(1)) \tfrac{\rho}{t})  \lesssim C\, \frac{t^2}{\rho^2}.
	\end{align}
	If $\tfrac{\rho}{t} < 0.9$, we can instead apply estimate \eqref{estimate_2_X} from Lemma \ref{improved_estimates_X} to get
	\begin{align}
		\Phi^{>\rho}_u(t) &\lesssim w \bbP(X(w) > (1+o(1)) \tfrac{\rho}{t}) \\
		&\lesssim \tfrac{t}{\rho}\cdot we^{-w/2} +  \big(\log \tfrac{t}{\rho}\big)^{5/2} \lesssim_\rho (\log t)^{5/2},
	\end{align}
	where in the last step we used the sub-linearity of $x^\pm$ as before.
\end{proof}

Using similar arguments, we next give
\begin{proof}[Proof of Lemma~\ref{l:5.5}] Recall the definition \eqref{e:A5.7} of $E_u(\rho)$ and unwind the calculations preceding it, it is easy to check that
\begin{align}
	& E_u(\rho) - \rme^{-\sqrt{2}u} \bbE \cE_\circ\big([-u, \infty); -\tfrac{1}{\sqrt{2}} \log u + [x^-_u, 0]\big) \\
	&= - \tfrac{1}{\sqrt{2}}\int_{t^-_u}^{1} t^{-1} dt \int_0^\infty y\mathbbm{1}_{ty > \rho} u \bbP(Y_u(x(t))\in dy)
+ \tfrac{1}{\sqrt{2}}
\int_{1}^{t^+_u} t^{-1} dt \int_0^\infty y \mathbbm{1}_{ty\leq \rho} \ u \bbP(Y_u(x(t))\in dy) \,.
\end{align}
We claim that as $u\to \infty$, the first term tends to
\begin{equation}\label{eq:l:5.5_integral_>_rho}
	-\tfrac{1}{\sqrt{2}}\int_{0}^{1} t^{-1} dt \int_0^\infty y\mathbbm{1}_{ty > \rho} \Gamma(\rmd y)  \,,
\end{equation}
and the second term tends to
\begin{equation}\label{eq:l:5.5_integral_<_rho}
	\tfrac{1}{\sqrt{2}} \int_{1}^{\infty} t^{-1} dt \int_0^\infty y\mathbbm{1}_{ty \leq \rho} \Gamma(\rmd y) \,,
\end{equation}
and both limits are finite thanks to~\eqref{e:A5.15}, which add up to the constant $C_\rho$ in the Lemma~\ref{l:5.5}. The proof strategy for these claims is the same as for Lemma~\ref{stable_main_convergence}: we just need to check that the inner integrals converge and satisfy the bounds
\begin{equation}\label{eq:l:5.5_integral_>_rho_bound}
	\int_0^\infty y\mathbbm{1}_{ty > \rho} u \bbP(Y_u(x(t))\in dy) \lesssim \frac{t}{\rho}
\end{equation}
and for $t\in [1, t^+_u]$,
\begin{equation}\label{eq:l:5.5_integral_<_rho_bound}
	\int_0^\infty y \mathbbm{1}_{ty\leq \rho} \ u \bbP(Y_u(x(t))\in dy) \lesssim_\rho t^{-1}(\log t)^{5/2},
\end{equation}
and the conclusions then follow from Dominated Convergence (using $w \sim u$ and $t^+_u = e^{o(u)}$). The bound \eqref{eq:l:5.5_integral_>_rho_bound} follows from estimate \eqref{estimate_0_X_1}, and the bound \eqref{eq:l:5.5_integral_<_rho_bound} follows from estimate \eqref{estimate_1_X} in the same way as in proof of Lemma \ref{lem:bound_inner_integral} for $\Phi^{\leq\rho}_u(t)$. The convergence of the inner integral in \eqref{eq:l:5.5_integral_>_rho} follows from the weak convergence \eqref{convergence_conjecture_applied} together with a truncation argument like in proof of Lemma \ref{lem:convergence_inner_integral} for $\Phi^{\leq\rho}_u(t)$, except here we truncate at $ty < M$ for $M \gg \rho$ and use the bound \eqref{eq:l:5.5_integral_>_rho_bound} to control the error uniformly. Finally, the convergence of the inner integral in \eqref{eq:l:5.5_integral_<_rho} can be shown in the same way as in proof of Lemma \ref{lem:convergence_inner_integral} for $\Phi^{\leq\rho}_u(t)$. 
\end{proof}

\ifonlystable
\section{The intensity of unusually large clusters}
\label{s:A4}
\fi

We now give the proof of Theorem~\ref{stable_key_lemma}.
To this end, we shall use the weak representation of the cluster law from Subsection~\ref{s:2.1}. 
\ifonlystable
\subsection{Reduction to a DRW event}
\fi
Recalling the notation in Subsection~\ref{s:2.1}, for a Borel set $B$,  define
	\[
	\cC_{t, B}^*(\cdot ):= \int_B \cE_{s}^{s} \big(\cdot - \wh{W}_{t,s} \big)\cN(\rmd s).
	\]
and abbreviate $\cC_t^*(\cdot) := \cC_{t, [0, t/2]}^*(\cdot)$. We shall show that there exists a non-increasing function $H: (0,\infty) \to [0,\infty)$ such that for all $x > 0$ that is a continuity point of $H$, 
	\begin{equation}
		\label{eq:H_tv}
		H_{t, v}(x) := v \bbP_{0,0}^{t,0} \bigg( \frac{\cC^*_t([-v, 0])}{ve^{\sqrt{2}v}} > x  
		\ \Big|\, \cA_t \bigg) \longrightarrow H(x) 
	\end{equation}
as $t \to \infty$ followed by $v \to \infty$ and that 
\begin{equation}\label{eq:growth_bound_H}
H(x) \lesssim x^{-2} \wedge  (\log x^{-1})^{5/2}.
\end{equation} 
In view of Lemma~\ref{l:7.0} and Lemma~\ref{l:5.2}, this will imply the first two parts of Theorem~\ref{stable_key_lemma}.

For $r \in [0,t/2]$ we start by decomposing the event $\cA_t$ from~\eqref{e:2.5} according to where the DRW is highest in $[r,t/2]$. That is, we write 
\[
1_{\cA_t} = \int_{r}^{t/2} 1_{\cA_{t, r}(s)} \, \cN(ds)
\quad \text{a.s.} \,,
\]
where for $s \in [r, t/2]$,
\begin{equation}
 \cA_{t, r}(s) := \{\forall \sigma\in \cN_{[r, t/2]}: \wh{W}_{t,\sigma} + \wh{h}^{\sigma*}_{\sigma} \leq \wh{W}_{t,s} + \wh{h}^{s*}_{s} \} \cap \cA_t \,.
\end{equation}
Using Palm Calculus (see, e.g., \cite[Chapter 13]{daley2007introduction}) we can write the left side of \eqref{eq:H_tv} as 
\begin{equation}
H_{t,v}(x) = \frac{v}{\bbP_{0,0}^{t,0} (\cA_t)} \int_{r}^{t/2} \bbP_{0,0}^{t,0} \bigg(\frac{\cC^*_t([-v, 0])}{ve^{\sqrt{2}v}} > x  \
	; \, \cA_{t, r}(s) \Big| s \in \cN \bigg) \cdot 2ds 
\end{equation}

The probability in the denominator is $(C+o(1)) t^{-1}$ as $t\to \infty$ by 
Lemma~\ref{lem:15}. At the same time, by conditioning on $\wh{W}_{t,s}, \wh{h}^{s*}_{s}$, and $\cE_s^{s}$, the probability inside the integral is equal to
\begin{multline}
\int \frac{1}{\sqrt{2\pi s(1-s/t)}} e^{-\frac{(w+\gamma_{t,s})^2}{2 s(s/t)}} dw  
\times \int_{z\leq -w} \bbP\left(\wh{h}^*_{s} \in dz, \frac{\cE_s([-v, \infty) - w)}{ve^{\sqrt{2}(v+w) - \frac{v^2}{2s} }} \in dy\right) \\
	\quad \quad  \times \bbP_{0,0}^{t,0} \bigg(\frac{\cC^*_t([-v, 0])}{ve^{\sqrt{2}(v+w) - \frac{v^2}{2s} }} > \tilde x -y  \
	; \cA_t(w+z; r, t/2) \cap \cA_{t} \ \Big|\, \wh{W}_{t,s} = w\bigg) \,,
\end{multline}
where we henceforth abbreviate:
\begin{equation}
\label{e:A5.56}
\tilde x = x e^{-\sqrt{2}w + \frac{v^2}{2s}} \,.
\end{equation}

Introducing for brevity
\begin{equation}
	f_{t,s}(w) := \frac{1}{\sqrt{2\pi (1-s/t)}} e^{-\frac{(w+\gamma_{t,s})^2}{2 s(1-s/t)}} \,,
\end{equation}
\begin{equation}
	\mu_{v, s, w}(dy, dz) := \bbP\left(\wh{h}^*_{s} \in dz, \frac{\cE_s([-v, \infty) - w)}{ve^{\sqrt{2}(v+w) - \frac{v^2}{2s} }} \in dy\right) 
\end{equation}
and
\begin{equation}\label{eq:Gtrvsw_def}
	G_{t, v, r, s, w,x}(y, z) := \frac{1}{s^{-1}t^{-1}} \bbP_{0,0}^{t,0} \bigg(\frac{\cC^*_t([-v, 0])}{ve^{\sqrt{2}(v+w) - \frac{v^2}{2s} }} > \wt{x}- y  \
	; \cA_t(w+z; r, t/2) \cap \cA_{t} \ \Big|\, \wh{W}_{t,s} = w\bigg) \,,
\end{equation}

we thus get 
\begin{equation}
\label{e:A5.60}
\begin{split}
	H_{t, v}(x) & = \big(1+o(1)\big) \int_{r}^{t/2} v\cdot s^{-3/2} \rmd s \int \, f_{t,s}(w)
	 \,\rmd w \int_{z\leq -w} G_{t, v, r, s, w,x}(y, z) \, \mu_{v, s, w}(dy, dz) \\
	 &= \int_{r/v^2}^{t/2v^2} s'^{-3/2} \rmd s' \int \, f_{t,s}(w)
	 \,\rmd w \int_{z\leq -w} G_{t, v, r, s, w,x}(y, z) \, \mu_{v, s, w}(dy, dz) \,,
\end{split}
\end{equation}
as $t \to \infty$ uniformly in $v, r$, where we performed the change of variable 
\begin{equation}
\label{e:A5.61}
s = s' v^2 \,,
\end{equation}
to cancel the factor $v$. Above and henceforth we assume that $s$, $s'$ and $v$ are always related as in~\eqref{e:A5.61}, so that in particular
\begin{equation}
\label{e:A5.63}
	\tilde x = x \rme^{-\sqrt{2} w+\frac{1}{2s'}} \,.
\end{equation}
Also, from now on, we shall fix some function $r = r_v: \bbR_+ \to \bbR$ such that as $v \to \infty$,
\begin{equation}
\label{e:A5.62}
	r_v \to \infty, \qquad r_v = o(v^2) \,.
\end{equation}

Now, $f_{t,s}(w)$ is an explicit function and asymptotics for $\mu_{v,s,w}$ in the regime of interest $s=s'v^2$ are given by Proposition~\ref{prop:6.5}. It thus remains to control $G_{t, v, r_v, s, w,x}$. To this end, for $R > 0$, we first define a truncated version $G^R_{t, v, r, s, w,x}(y, z)$ by the formula in \eqref{eq:Gtrvsw_def}, only with $\cC_{t}^*(\cdot)$ replaced by
\begin{equation}
\cC_{t, [s-R, s+R]}^*(\cdot) := \int_{s-R}^{s+R} \cE_{s}^{s} \big(\cdot - \wh{W}_{t,s} \big)\cN(\rmd s) \,.
\end{equation}
That is, we only look at the contribution to the cluster from BBMs which branched off the spine at time $s + \Theta(1)$, where the DRW is highest. 

First, we need to verify that $G^R_{t, v, r_v, s, w,x}$ is a good approximation for
$G_{t, v, r_v, s, w,x}$. This is included in
\begin{lem}\label{lem:G^R_v_properties} 
Let $R > 0$, $x > 0$ and abbreviate $G_v \equiv G_{t, v, r_v, s, w,x}$ and $G^R_v \equiv G^R_{t, v, r_v, s, w,x}$. Then, for all $v,t > 0$ such that $r_v < t/2$ and $z\leq -w$,
	\begin{enumerate}[label=(\roman*)]
		\item\label{G^R_v_monotone} Monotonicity: $G^R_v(y, z),  G_v(y, z)$ are increasing in $y, z, R$.
		\item\label{G^R_v_bound} Uniform bound: if $r_v \leq s\leq t/2$, $0\leq G_v(y, z) \lesssim (w^-+1)^2$ 
		\item\label{G^R_v_approximation} Approximation: if $r_v \leq s \leq t/2-R$, then for any $\eps > 0$,
		\begin{equation}
			G^R_v(y, z) \leq G_v(y, z) \leq G^R_v(y+\eps, z) + \eps^{-1} C(s', w)  [v^{-1} + R^{-1/2}]
		\end{equation}
		where $C(s', w)$ is a constant only depending on $s', w$. 
	\end{enumerate}
\end{lem}
The proof of this lemma, which is technical while not illuminating, is deferred to Subsection~\ref{s:G^R_v_properties}.

Next, we need that $G^R_{t, v, r_v, s, w}$ indeed admits asymptotics in the limits of the desired statement. This is the content of the next key proposition.
\begin{prop}
\label{p:A11}
For all $x > 0$, $w \in \bbR$, $s' > 0$, there exists a function $G^R_{s', w, x} : \bbR_+ \times \bbR \to \bbR_+$ such that for all $y \geq 0$, $z \in \bbR$, 
as $t\to \infty$ followed by $v \to \infty$,
\begin{equation}
	G^R_{t, v, r_v, s, w,x}(y,z) \longrightarrow G^R_{s', w, x}(y,z) \,,
\end{equation}
with $s$ on the left hand side determined according to~\eqref{e:A5.61}.
The function $G^R_{s', w, x}$ is jointly continuous on the half plane $\{(y, z): y\leq \tilde x\}$ and is a continuous function of $z$ only for $y > \tilde x$.
\end{prop}
The proof of this proposition will be given in Subsection~\ref{s:proof_p:A11}.

\ifonlystable
\subsection{Proof of Theorem~\ref{stable_key_lemma}}
\fi
Assuming Lemma~\ref{lem:G^R_v_properties} and Proposition~\ref{p:A11}, we can pass to the limit in~\eqref{e:A5.60}. This is accomplished via the following two technical lemmas, providing both asymptotics and a dominating function. The proofs of these two results, which do not add any new conceptual ideas, are deferred to Subsection~\ref{s:integral_G_mu} and~\ref{s:HM_eta_approximation}, respectively.

\begin{lem}[Asymptotics]
\label{lem:integral_G_mu} 
Let
\begin{equation}\label{eq:H_tvr}
	H_{t, v, r, x}(s', w) := 1_{[r,t/2]}(s) f_{t,s}(w) \int_{z\leq -w} G_{t, v, r, s, w,x}(y, z) \, \mu_{v, s, w}(dy, dz) \,.
\end{equation}
and
\begin{equation}
H_x(s', w) := \int G_{s',w,x}(y, z) \mu(dy, dz) 
\quad ; \qquad G_{s',w,x}(y,z) := \lim_{R\to \infty} G^R_{s',w,x}(y,z) \,.
\end{equation}
Then $H_x(s', w)$, is well defined, non-increasing in $x$ and for all $x \geq 0$, $s' > 0$, $w \in \bbR$ and $\epsilon > 0$,
	\begin{equation}
		 \liminf_{v\to \infty} \liminf_{t \to \infty} H_{t, v, r_v,x}(s', w) \geq H_x(s', w) \,,
	\end{equation}
	and
	\begin{equation}
	\limsup_{v\to \infty} \limsup_{t \to \infty} H_{t, v, r_v,x}(s', w) \leq H_{x-\epsilon}(s', w) \,.
	\end{equation}
\end{lem}
\begin{lem}[Dominating function]
\label{lem:HM_eta_approximation} 
For $M, \eta > 0$, let $H^{M, \eta}_{t, v}$, $G^{M, \eta}_{t, v, r, s, w,x}$
and $H^{M, \eta}_{t, v, r,x}$ be defined by formulas~\eqref{eq:H_tv}, \eqref{eq:Gtrvsw_def} and~\eqref{eq:H_tvr} respectively, along 
with $\cC_t^*(\cdot)$ replaced by
\[
\cC_{t}^{*M, \eta}(\cdot) := \int_{\eta v^2}^{\eta^{-1} v^2} \cE_{s}^{s} \big(\cdot - \wh{W}_{t,s} \big) 1_{\wh{W}_{t,s} \leq M} \, \cN(\rmd s).
\]
If $\eta^{-1} v^2 \leq t/2$, then
	\begin{equation}
		H^{M, \eta}_{t, v}(x) \leq H_{t, v}(x) \leq H^{M, \eta}_{t, v}(x-\eps) + \eps^{-1} [ve^{-\sqrt{2} v} + e^{-C'M} + \sqrt{\eta}]
	\end{equation}
Moreover, for $t \geq 100 \eta^{-1} v^2$, $r \in [v^{\alpha}, \eta v^2/10]$
with any $\alpha < 2$, and $v \geq M$,
	\begin{equation}\label{eq:HM_eta_dominating}
		H^{M, \eta}_{t, v, r,x}(s', w) \leq C(\eta) x^{-1} e^{-c|w|} [s'^{\frac{1}{2-\alpha}} 1_{s'\leq \eta/2} + 1_{s' \geq \eta/2}]
	\end{equation}
for some absolute constant $c > 0$ and a constant $C(\eta)$ that only depends on $\eta$. 	
\end{lem}

We are now in a position to prove the theorem. 
\begin{proof}[Proof of Theorem~\ref{stable_key_lemma}]
By~\eqref{e:A5.60} and~\eqref{eq:H_tvr} we have
\begin{equation}
\label{e:A5.72}
H_{t, v}(x) = \big(1+o(1)\big) \int_0^\infty s'^{-3/2} ds'\int dw \cdot H_{t, v, r,x}(s', w) \,.
\end{equation}
Thanks to Lemma \ref{lem:integral_G_mu}, we then immediately conclude by Fatou that,
\begin{equation}\label{H_fatou_lower}
\liminf_{v\to \infty} \liminf_{t \to \infty} H_{t, v}(x) \geq C \int_0^\infty s'^{-3/2} ds'\int dw \cdot H_x(s', w) =: H(x) \,.
\end{equation}
$H(x)$ defined above is clearly non-negative and non-increasing in $x$. 

For the complementary upper bound, for $M, \eta > 0$, the truncated analog of~\eqref{e:A5.72} is
\begin{equation}
\label{e:A5.72a}
H^{M, \eta}_{t, v}(x) = \big(1+o(1)\big) \int_0^\infty s'^{-3/2} ds'\int dw \cdot H^{M, \eta}_{t, v, r}(s', w) \,.
\end{equation}
Since replacing $H^{M, \eta}_{t, v, r}(s', w)$ above by its upper bound in \eqref{eq:HM_eta_dominating} from Lemma~\ref{lem:HM_eta_approximation}, yields a convergent integral, we can (for example) choose $r_v = v^\alpha$ with $\alpha = 1$ and apply Reverse Fatou to obtain for any $\epsilon > 0$,
\begin{align}
\limsup_{v\to \infty} \limsup_{t \to \infty} H^{M, \eta}_{t, v}(x) &\leq C \int_0^\infty s'^{-3/2} ds'\int dw \cdot \limsup_{v\to \infty} \limsup_{t \to \infty}  H^{M, \eta}_{t, v, r_v,x}(s', w) \\
	&\leq C \int_0^\infty s'^{-3/2} ds'\int dw \cdot \limsup_{v\to \infty} \limsup_{t \to \infty}  H_{t, v, r_v,x}(s', w) \\
	&\leq C \int_0^\infty s'^{-3/2} ds'\int dw \cdot H_{x-\epsilon}(s', w) = H(x -\eps)
\end{align}
where we used the upper bound from Lemma \ref{lem:integral_G_mu}. 

Thanks to the first part of Lemma~\ref{lem:HM_eta_approximation} we have 
\begin{equation}\label{eq:HM_eta_approximation}
	\limsup_{v\to \infty} \limsup_{t \to \infty} H_{t, v}(x) \leq \sup_{M, \eta > 0} \limsup_{v\to \infty} \limsup_{t \to \infty} H^{M, \eta}_{t, v}(x-\eps)
\end{equation}
so that altogether,
\begin{equation}\label{H_fatou_upper}
	\limsup_{v\to \infty} \limsup_{t \to \infty} H_{t, v}(x) \leq H(x - 2\eps)
\end{equation}
Taking $\eps \to 0$ and combining \eqref{H_fatou_upper} with \eqref{H_fatou_lower} gives us that
\begin{equation}
	\limsup_{v\to \infty} \limsup_{t \to \infty} |H_{t, v}(x) - H(x)| \to 0
\end{equation}
for every continuity point $x$ of $H$. This shows the convergence \eqref{eq:H_tv} and consequently~\eqref{convergence_conjecture}.

For the second part of the theorem, we use the layer cake representation to write for any continuity point $\delta$ of $\Gamma$,
	\begin{align}
		w \bbE \left(X(w); X(w) > \delta \right) &= \int_\delta^\infty w \bbP(X(w) > y) dy + \delta \cdot w \bbP(X(w) > \delta) \\
		&\to \int_\delta^\infty \Gamma((y, \infty)) dy + \delta \cdot \Gamma((\delta, \infty)) = \int y \mathbbm{1}_{y > \delta} \Gamma(dy)
	\end{align}
	as $w\to \infty$, by Dominated Convergence, using the established convergence \eqref{convergence_conjecture} and the bound \eqref{estimate_0_X_2}. The bound \eqref{estimate_1_X} then allows us to take $\delta \to 0$, and~\eqref{eq:limiting_measure_1st_moment} follows from~\eqref{estimate_0_X_0_1}.
	
Lastly, the bound \eqref{eq:growth_bound_H} and consequently~\eqref{e:A5.15} 
are immediately inherited from the bounds \eqref{estimate_0_X_2} and \eqref{estimate_2_X}.
\end{proof}

\ifonlystable
\subsection{Proof of Proposition~\ref{p:A11}}
\fi
\label{s:proof_p:A11}

We begin by applying the Markov property of the DRW at time $s$ and condition at times $s-R, s+R$ to express $G^R_{t, v, r_v, s, w,x}(y, z)$ for $z\leq -w$ as
\begin{align}\label{eq:convergence_cJ_leq_R_rewrite}
	& \int \bbP_{0, 0}^{s, w}(\wh{W}_{t, s-R} \in w + d\xi_-) \times \bbP_{s, w}^{t, 0}(\wh{W}_{t, s+R} \in w + d\xi_+)  \\
	& \quad \times \frac{1}{s^{-1}}\bbP_{0, 0}^{s-R, w+\xi_-}\left(\cA_t(w+z;\, r, s-R) \cap  \cA_t(0, r)\right) \label{eq:convergence_cJ_leq_R_rewrite_2} \\
	& \quad \times \frac{1}{t^{-1}}\bbP_{s+R, w+\xi_+}^{t, 0}\left(\cA_t(w+z;\, s+R, t/2) \cap  \cA_t(t/2, t)\right) \label{eq:convergence_cJ_leq_R_rewrite_3} \\
	& \quad \times \bbP_{s-R,w+\xi_-}^{s+R, w+\xi_+} \left(\frac{\cC^*_{t, [s-R, s+R]}([-v, 0])}{ve^{\sqrt{2}(v+w) - \frac{v^2}{2s} }} > \tilde x -y  \
	; \cA_t(w+z; s-R, s+R) \Big|\, \wh{W}_{t,s} = w\right). \label{eq:convergence_cJ_leq_R_rewrite_4} 
\end{align}
By direct computation, the law of $\wh{W}_{t, s-R} -w$ under $\bbP_{0, 0}^{s, w}$ is Gaussian with mean $o(1)$ and variance $\sim R$ as $s\to \infty$, and the law of $\wh{W}_{t, s+R}-w$ under $\bbP_{s, w}^{t, 0}$ is Gaussian with mean $o(1)$ and variance $\sim R$ as $s, t-s \to \infty$.  Thus, as $t\to \infty$ followed by $v\to \infty$ (recall $s = s'v^2$)
\[
\bbP_{0, 0}^{s, w}(\wh{W}_{t, s-R} \in w + d\xi_-)/d\xi_- \times \bbP_{s, w}^{t, 0}(\wh{W}_{t, s+R} \in w + d\xi_+)/d\xi_+ \to \frac{1}{\sqrt{2\pi R}} e^{-\frac{|\xi_-|^2}{2R}}  \times \frac{1}{\sqrt{2\pi R}} e^{-\frac{|\xi_+|^2}{2R}} 
\]
and moreover these joint densities are uniformly bounded by some Gaussian function. Next, by the asymptotics in Lemma \ref{lemma_mixed_ballot}, the product of \eqref{eq:convergence_cJ_leq_R_rewrite_2} and \eqref{eq:convergence_cJ_leq_R_rewrite_3} is asymptotic to 
\[
C f(\xi_+ - z) g(\xi_- - z)
\] 
as $t\to \infty$ followed by $v\to \infty$. It remains to estimate the probability in \eqref{eq:convergence_cJ_leq_R_rewrite_4}. 
 Applying time shift and tilting this can be rewritten as
\begin{align}\label{eq:convergence_cJ_leq_R_rewrite_last_probability}
	\bbP_{-R,\xi_-}^{R, \xi_+} \Big(\int_{-R}^{R} \frac{\cE_{s+u}^{u} \big([-(v+w), \infty) -  (W_u - \eps_{u, s}) \big)}{ve^{\sqrt{2}(v+w) - \frac{v^2}{2s}} } \cN(\rmd u) > \tilde x -y; \max\limits_{\substack{u \in \cN \\ |u| \leq R}} \big(W_u - \eps_{u, s} + \wh{h}^{u*}_{s+u} \big) \leq z\Big)
\end{align}
for a two-sided Brownian motion $W$ starting at $W_0 = 0$ and $\eps_{u, s}$ is a deterministic drift uniformly $o(R/s)$ for $u\in [-R, R]$. We claim that as $v\to \infty$ this probability (notice it does not involve $t$) converges to
\begin{equation}\label{eq:convergence_cJ_leq_R_rewrite_last_probability_limit}
\bbP_{-R,\xi_-}^{R, \xi_+} \bigg(\int_{-R}^{R} c_0 e^{\sqrt{2} W_u} Z^{(u)} \cN(\rmd u) > \tilde x -y; \ \max\limits_{u \in \cN_{[-R, R]}} \, \big(\wh{h}_\infty^{*(u)} + W_u\big) \leq z \bigg). 
\end{equation}
where $(Z^{(u)}, \wh{h}_\infty^{*(u)})$ are i.i.d.~copies of $(Z, \wh{h}_\infty^{*})$. We can first show the convergence of the probabilities conditional on $\cN_{[-R, R]}$ and then integrate over that. Conditional on $\cN_{[-R, R]}$, it follows from Proposition \ref{prop:6.5} that the integral and maximum in \eqref{eq:convergence_cJ_leq_R_rewrite_last_probability} jointly converge in distribution to those in \eqref{eq:convergence_cJ_leq_R_rewrite_last_probability_limit} (by first conditioning on $(W_u: u\in \cN_{[-R, R]})$ and then integrating over that). The convergence of the probabilities (conditional on $\cN$) would now follow from Portmanteau theorem if, in the limit, the equality cases a.s.~do not happen (conditional on $\cN$). This is the case as long as $\cN$ has at least one point in $(-R, R)\setminus\{0\}$
, because $Z$ is a.s.~nonzero and the finite-dimensional laws of $(W_u: u\in (-R, R)\setminus\{0\})$ under $\bbP_{-R,\xi_-}^{R, \xi_+}$ are absolutely continuous w.r.t.~the Lebesgue measure. On the other hand, if $\cN$ has no point in $(-R, R)\setminus\{0\}$, it also a.s.~has no point in $[-R, R]$, in which case the probability (conditional on $\cN$) is always $1_{\tilde x - y < 0}$ and thus trivially converges. Altogether, by Dominated Convergence, we have that $G^R_{t, v, r_v, s, w,x}(y, z)$ converges to
	\begin{align}\label{eq:convergence_cJ_leq_R}
	G^R_{s',w,x}(y, z) &:= \int \bbP(W_R \in d\xi_+, W_{-R} \in d\xi_-)  \times C f(\xi_+  - z) g(\xi_- - z)  \\
	& \quad \times  \bbP_{-R, \xi_-}^{R, \xi_+} \bigg(\int_{-R}^{R} c_0 e^{\sqrt{2} W_u} Z^{(u)} \cN(\rmd u) > \tilde x -y; \ \max\limits_{u \in \cN_{[-R, R]}} W_u + \wh{h}_\infty^{*(u)} \leq z \bigg) 
\end{align}
as $t \to \infty$ followed by $v\to \infty$. The continuity properties claimed for $G^R_{s',w',x}$ then follow from Dominated Convergence using the fact that $f$ and $g$ are monotone and so have at most countably many discontinuity points, the finite-dimensional laws of $W$ (excluding times being conditioned) are absolutely continuous w.r.t.~the Lebesgue measure, and $Z$ is non-negative valued. This concludes the proof of Proposition \ref{p:A11}. 

\ifonlystable
\section{Proofs of technical lemmas}
\label{s:A5}
\fi

Let us now turn to all the technical lemmas whose proofs we have deferred. 
We begin with:
\ifonlystable
\subsection{Proof of Lemma~\ref{improved_estimates_X}}
\fi
\begin{proof}[Proof of Lemma~\ref{improved_estimates_X}]
\ifonlystable
In Lemma~3.6 of~\cite{hartung2024growth}
\fi
take $u = v$ and $M_u, \eta_u^{-1} = c_1 \log \delta^{-1}$ for large enough $c_1$ ($10 \max(17, 1/C')$ will do), so that
	\begin{equation}\label{atypical_event_est_applied}
		\wt{\bbP} \Big(\cB_{v}(t) \Big|\, \wh{h}^*_t = \wh{h}_t(X_t) = 0\Big) \lesssim v^{-1} (\log \delta^{-1})^{\frac{5}{2}},
	\end{equation}
	and
	\begin{equation}
		\label{1st_moment_est_applied}
		\wt{\bbE}\Big(\cC_{t,r_t}^*\big([-v,0]\big) ;\; \cB_{v}(t)^{\rmc} \, \Big| \,  \wh{h}^*_t =  \wh{h}_t(X_t) = 0\Big)  \lesssim  \rme^{\sqrt{2}v} \left(v \rme^{-v/2} +  \delta^{10} + v t^{-1/2}\right).
	\end{equation}
	To derive \eqref{estimate_1_X}, we abbreviate by $\cA_v(t)$ the event
	\[
	\cC_{t,r_t}^*\big([-v,0]\big) > \delta v e^{\sqrt{2} v}.
	\]
	Now by union bound
	\begin{multline}
	\wt{\bbE}\Big(\cC_{t,r_t}^*\big([-v,0]\big) ;\; \cA_v(t)^{\rmc} \, \Big| \,  \wh{h}^*_t =  \wh{h}_t(X_t) = 0\Big) 
		 \leq \wt{\bbE}\Big(\cC_{t,r_t}^*\big([-v,0]\big) ;\; \cB_v(t)^{\rmc} \, \Big| \,  \wh{h}^*_t =  \wh{h}_t(X_t) = 0\Big)\\
+ \wt{\bbE}\Big(\cC_{t,r_t}^*\big([-v,0]\big) ;\; \cB_v(t) \cap \cA_v(t)^{\rmc} \, \Big| \,  \wh{h}^*_t =  \wh{h}_t(X_t) = 0\Big) 
\end{multline}
The right hand side can be further bounded by a constant times
\begin{multline}
\rme^{\sqrt{2}v} \left(v \rme^{-v/2} +  \delta^{10} + v t^{-1/2}\right) + \delta v e^{\sqrt{2} v} \bbP\Big(\cB_{v}(t) \Big|\, \wh{h}^*_t = \wh{h}_t(X_t)\Big)  \\
		\lesssim \rme^{\sqrt{2}v} \left(v \rme^{-v/2} +  \delta (\log \delta^{-1})^{\frac{5}{2}} + v t^{-1/2}\right) 
\end{multline}
	After taking time $t\rightarrow \infty$ we get \eqref{estimate_1_X}. Derivation of \eqref{estimate_2_X} is similar. By Markov inequality,
	\begin{align}
		\wt{\bbP}\Big(\cA_v(t) \cap \cB_v(t)^{\rmc} \, \Big| \,  \wh{h}^*_t =  \wh{h}_t(X_t) = 0\Big) &\leq \delta^{-1} v^{-1} e^{-\sqrt{2} v}\wt{\bbE}\Big(\cC_{t,r_t}^*\big([-v,0]\big) ;\; \cB_v(t)^{\rmc} \, \Big| \,  \wh{h}^*_t =  \wh{h}_t(X_t) = 0\Big) \\
		&\lesssim \delta^{-1} \left( \rme^{-v/2} +  \delta^{10} v^{-1} +  t^{-1/2}\right)
	\end{align}
	and so by union bound
	\begin{align}
		\wt{\bbP}\Big(\cA_v(t) \Big| \,  \wh{h}^*_t =  \wh{h}_t(X_t) = 0\Big) &\leq \wt{\bbP}\Big(\cA_v(t) \cap \cB_v(t)^{\rmc} \, \Big| \,  \wh{h}^*_t =  \wh{h}_t(X_t) = 0\Big) + \wt{\bbP}\Big(\cB_v(t) \Big| \,  \wh{h}^*_t =  \wh{h}_t(X_t) = 0\Big) \\
		&\lesssim \delta^{-1} \left( \rme^{-v/2} +  \delta^{10} v^{-1} +  t^{-1/2}\right) + v^{-1} (\log \delta^{-1})^{\frac{5}{2}} \\
		&\lesssim \delta^{-1} \left( \rme^{-v/2} + t^{-1/2}\right) + v^{-1} (\log \delta^{-1})^{\frac{5}{2}}
	\end{align}
	After taking time $t\rightarrow \infty$ we get \eqref{estimate_2_X}. 
\end{proof}

\ifonlystable
\subsection{Proof of Lemma~\ref{lem:G^R_v_properties}}
\fi
\label{s:G^R_v_properties}
Henceforth, to economize on notation, we shall use the abbreviations:
\begin{equation}
G_v \equiv G_{t, v, r_v, s, w,x} \quad, \qquad G^R_v \equiv G^R_{t, v, r_v, s, w, x}
\quad , \qquad G^R \equiv G^R_{s', w, x}
\quad , \qquad \mu_v \equiv \mu_{v, s, w} \,.
\end{equation}

\begin{proof}[Proof of Lemma~\ref{lem:G^R_v_properties}]
Part \ref{G^R_v_monotone} is obvious from the way $G_v^R, G_v$ are defined. Part \ref{G^R_v_bound} follows from the ballot upper bounds \eqref{e:52} in Lemma \ref{lem:15} applied to
	\[
	G_v \leq \frac{1}{s^{-1}t^{-1}} \bbP_{0,0}^{t,0} \bigg(\cA_{t} \ \Big|\, \wh{W}_{t,s} = w\bigg)  = \frac{1}{s^{-1}} \bbP_{0,0}^{s,w} (\cA_{t}(0, s))  \frac{1}{t^{-1}} \bbP_{s,w}^{t,0} (\cA_{t}(s, t)).
	\]
Lastly, part \ref{G^R_v_approximation} would follow from a straightforward application of Union Bound and Markov inequality, provided the moment bound in Lemma~\ref{lemma_moment_estimate_for_cluster_outside_R} below,
for the contribution to the cluster from BBMs which branched from the spine at a time outside ${[s-R, s+R]}$.
\end{proof}.
\begin{lem}
\label{lemma_moment_estimate_for_cluster_outside_R} 
For $s = s' v^2$, \begin{equation}\label{eq:small_cJ_>R}
		\frac{1}{s^{-1}t^{-1}} \bbE_{0,0}^{t,0} \bigg(\frac{\cC_{t}^*\setminus \cC_{t, [s-R, s+R]}^*([-v, 0])}{ve^{\sqrt{2}(v+w) - \frac{v^2}{2s} }}  \
		; \cA_{t} \ \Big|\, \wh{W}_{t,s} = w\bigg) \leq C(s', w) [v^{-1} + R^{-1/2}].
	\end{equation}
\end{lem}
\begin{proof}
By applying Markov property of the DRW at time $s$ and the upper bounds from Lemma \ref{lemma_mixed_ballot}, the left side of \eqref{eq:small_cJ_>R} is at most $C(w^-+1)$ times
	\begin{align}\label{eq:small_cJ_>R_bound}
		\frac{1}{s^{-1}}\bbE_{0,0}^{s,w} \bigg(\frac{\cC_{t, [0, s-R]}^*([-v, 0])}{ve^{\sqrt{2}(v+w) - \frac{v^2}{2s} }}  \
		; \cA_{t}(0, s)\bigg) + \frac{1}{t^{-1}}\bbE_{s,w}^{t,0} \bigg(\frac{\cC_{t, [s+R, t/2]}^*([-v, 0])}{ve^{\sqrt{2}(v+w) - \frac{v^2}{2s} }}  \
		; \cA_{t}(s, t)\bigg).
	\end{align}
	By Palm calculus, the first expectation equals $\displaystyle\int_0^{s-R} \wh{j}_{s, v}(\sigma) \cdot 2\rmd \sigma$ with
	\begin{equation}
		\wh{j}_{s, v}(\sigma) := v^{-1} \rme^{-\sqrt{2} (v+w) +\frac{v^2}{2s}} \, \bbE_{0, 0}^{s, w}\left[\cE^{\sigma}_{\sigma}([-v, 0] - \wh{W}_{s, \sigma}) 1_{\wh{h}_{\sigma}^{\sigma*}\leq -\wh{W}_{s, \sigma}}1_{\cA_s}\right],
	\end{equation}
	where we replaced $\wh{W}_{t, \sigma}$ by $\wh{W}_{s, \sigma}$ for $\sigma\in [0, s]$ as they have the same law under the bridge conditioning $\bbP_{0,0}^{s, w}$. In the case $w = 0$, Lemma 5.5 in \cite{CHL19} tells us
	\begin{equation}
		\wh{j}_{s, v}(\sigma) \leq C e^{\frac{1}{2s'}} \, \frac{e^{-\frac{v^2}{16\sigma}} + e^{-v/2} }{s(\sigma+1)\sqrt{\sigma}} 
	\end{equation}
	for $0\leq \sigma\leq \frac{s}{2}$ (recall $s = s' v^2$). In this case, essentially the same proof there (in reversed time) also shows
	\begin{equation}
		\wh{j}_{s, v}(\sigma) \leq \frac{C e^{\frac{1}{2s'}}}{s(s-\sigma+1)\sqrt{s-\sigma}}
	\end{equation}
	for $\frac{s}{2} \leq \sigma\leq s$. By tilting $\wh{W}_{s, \sigma}$ (and then shifting $\wh{h}_{\sigma}^{\sigma*}$ to compensate), one can see these upper bounds also hold for any fixed $w$ but with a constant $C$ depending on $w$. Altogether, the first term in \eqref{eq:small_cJ_>R_bound} is at most
	\begin{align}
		\frac{1}{s^{-1}} \int_0^{s-R} \wh{j}_{s, v}(\sigma) \cdot 2\rmd \sigma
		&\leq C(w) e^{\frac{1}{2s'}} \, \left(\int_{0}^{\frac{s}{2}}  \frac{e^{-\frac{v^2}{16\sigma}} + e^{-v/2}}{(\sigma+1)\sqrt{\sigma}} \rmd \sigma + \int_{\frac{s}{2}}^{s-R} \frac{1}{(s-\sigma+1)\sqrt{s-\sigma}} \rmd \sigma \right) \\
		& \leq C(w) e^{\frac{1}{2s'}} \cdot C'[v^{-1} + e^{-v/2} + (R+1)^{-1/2}]
	\end{align}
	and analogous bound holds for the second term in \eqref{eq:small_cJ_>R_bound} by similar calculations. 
\end{proof}

\ifonlystable
\subsection{Proof of Lemma~\ref{lem:integral_G_mu}}
\fi
\label{s:integral_G_mu}

It is straightforward to check that, for fixed $s', w$ and $r_v = o(v^2)$, the pre-factors in front of the integral in \eqref{eq:H_tvr} converges to $1$ as $t \to \infty$ followed by $v \to \infty$. Moreover, it is immediate from Proposition \ref{prop:6.5} that the probability measures $\mu_{v} \underset{v \to \infty}\Longrightarrow  \mu$,  where $\mu$ is the law of the right hand side of~\eqref{e:6.47}.

By Part~\ref{G^R_v_monotone} of Lemma~\ref{lem:G^R_v_properties} and the continuity of the limit in Proposition~\ref{p:A11},  the function $G^R(y, z) \geq 0$ is jointly left continuous and increasing in $(y, z)$ and thus lower semicontinuous in $(y, z)$, so 
the Portmanteau theorem yields
\[
\liminf_{v\to \infty} \int_{z\leq -w} G^R(y, z) \mu_v(dy, dz) \geq\int_{z\leq -w} G^R(y, z) \mu(dy, dz)
\]
where we also used that $1_{z < -w}$ is lower semicontinuous and that the law of $\wh{h}^*_\infty$ has no atoms. For a matching upper bound, we can apply the Portmanteau theorem instead to the ``right-continuous modification'' $\lim_{\delta \to 0^+} G^R(y+\delta, z)$, which is well-defined as a monotone limit,  upper semicontinuous in $(y, z)$ since it is jointly right continuous and increasing in $(y, z)$, and moreover bounded from above by the uniform bound \ref{G^R_v_bound}, so 
\[
\limsup_{v\to \infty} \int_{z\leq -w} G^R(y, z) \mu_v(dy, dz) \leq\int_{z\leq -w} G^R(y+\tilde \eps, z) \mu(dy, dz)
\]
for any $\tilde \eps > 0$, where we also used that $1_{z \leq -w}$ is upper semicontinuous and the monotonicity bound $G^R(y, z) \leq \lim_{\delta \to 0^+} G^R(y+\delta, z) \leq G^R(y+\tilde \eps, z)$. \medskip 

Moreover, since the functions $G_v^R$ are increasing in $y, z$ and the limit $G^R$ is jointly continuous in the region $\{(y, z): y\leq \tilde x\}$, the convergence is necessarily uniform over bounded $(y, z)$ in this region. Similarly, for $y > \tilde x$, $G_v^R(y, z) \equiv G_v^R(z)$ is increasing in $z$ and the limit $G^R(y, z) \equiv G^R(z)$ is continuous in $z$, so the convergence is uniform over bounded $z$ in this region. Altogether, we have $G_v^R \to G^R$ uniformly over any compact set of $(y, z)$ as $t\to \infty$ followed by $v\to \infty$. Thus, by the tightness of $\mu_v$ and the uniform bound of $G_v^R, G^R$, we obtain
\[
\lim_{v\to \infty} \limsup_{t\to \infty} \int \big|G_v^R(y, z) - G^R(y, z)\big| \mu_v(dy, dz) \to 0.
\] 
Putting these together and using the approximation \ref{G^R_v_approximation} we get
\[
\liminf_{v\to \infty} \liminf_{t\to \infty} \int G_v(y, z) \mu_v(dy, dz) \geq\int G^R(y, z) \mu(dy, dz)
\]
and
\[
\limsup_{v\to \infty} \limsup_{t\to \infty} \int G_v(y, z) \mu_v(dy, dz) \leq \int G^R(y+2\tilde \eps, z) \mu(dy, dz) + \tilde \eps^{-1} C(w) R^{-1/2}.
\]
Finally, since we also have $G^R$ is increasing in $R$, defining
\begin{equation}
G(y,z) \equiv G_{s',w,x}(y,z) := \lim_{R\to \infty} G^R(y,z) \equiv \lim_{R\to \infty} G^R_{s',w,x}(y,z)
\end{equation}
 we have by monotone convergence
\[
\liminf_{v\to \infty} \liminf_{t\to \infty} \int G_v(y, z) \mu_v(dy, dz) \geq\int G(y, z) \mu(dy, dz)
\]
and by simply upper bounding
\[
\limsup_{v\to \infty} \limsup_{t\to \infty} \int G_v(y, z) \mu_v(dy, dz) \leq \int G(y+2\tilde \eps, z) \mu(dy, dz) \,.
\]
Setting
\begin{equation}
H_x(s', w) := \int G(y, z) \mu(dy, dz) \equiv \int G_{s',w,x}(y, z) \mu(dy, dz) 
\end{equation}
and observing that $G_{s',w,x}(y+2\tilde \eps, z) = G_{s',w,x-\eps}(y, z)$ for $\epsilon = 2\tilde \eps \rme^{\sqrt{2}w-(2s')^{-1}}$, in light of~\eqref{e:A5.63}, this concludes the proof of the lemma.

\ifonlystable
\subsection{Proof of Lemma~\ref{lem:HM_eta_approximation}}
\fi
\label{s:HM_eta_approximation}

Recall $H^{M, \eta}_{t, v, r}$ is defined by the formulas \eqref{eq:H_tvr} and \eqref{eq:Gtrvsw_def} with $\cC^{*M, \eta}_t$ instead of $\cC^{*}_t$, we have (after ``undoing''	 the conditioning on $\cE_s^{s}$),
\begin{align}\label{eq:H_M_eta_main}
H^{M, \eta}_{t, v, r}(\tilde x, s', w) \leq 1_{\left[\frac{r}{v^2}, \frac{t/2}{v^2}\right]}(s') \times 2 e^{-\frac{(w+\gamma_{t,s})^2}{2 s}} \times  \int_{z\leq -w} F^{M, \eta}_v(z) \, \bbP\big(\wh{h}^*_{s} \in dz\big)
\end{align}
with $s := s'v^2$, $\wt{x} := x e^{-\sqrt{2}w + \frac{v^2}{2s}}$, and
\[
F^{M, \eta}_v(z) := \frac{1}{s^{-1}t^{-1}} \bbP_{0,0}^{t,0} \bigg(\frac{\cC^{*M, \eta}_t([-v, 0])}{ve^{\sqrt{2}v }} > x 
; \cA_t(w+z; r, t/2) \cap \cA_{t} \ \Big|s \in \cN, \wh{W}_{t,s} = w, \wh{h}^*_{s} = z \bigg).
\]
We further decompose $\cC^{*M, \eta}_t([-v, 0]) = \cC_{< s}([-v, 0]) + \cC_{> s}([-v, 0]) + \cC_{= s}([-v, 0])$ with
\begin{align}
\cC_{< s}([-v, 0]) &= \int_{\eta v^2}^{(s'\wedge \eta^{-1}) v^2} \cE_{\sigma}^{\sigma} \big([-v, 0] - \wh{W}_{t,\sigma} \big) 1_{\wh{W}_{t,\sigma} \leq M} \, \cN(\rmd \sigma) \\
\cC_{> s}([-v, 0]) &= \int_{(s' \vee \eta) v^2}^{\eta^{-1} v^2} \cE_{\sigma}^{\sigma} \big([-v, 0] - \wh{W}_{t,\sigma} \big) 1_{\wh{W}_{t,\sigma} \leq M} \, \cN(\rmd \sigma) \\
\cC_{= s}([-v, 0]) &=  1_{s'\in [\eta , \eta^{-1}]} \, \cE_{s}^{s} \big([-v, 0] - w \big) \, 1_{w\leq M}.
\end{align}
Abbreviate 
\[\cB_{<s} := \cA_t(w+z; r, s) \cap \cA_{t}(0, r),  \ \ \ \cB_{>s} := \cA_t(w+z; s, t/2) \cap \cA_{t}(t/2, t).
\]
By Markov inequality, linearity of expectation, and the Markov property of the DRW, we have
\begin{align}
	F^{M, \eta}_v(z) \times x v e^{\sqrt{2}v} & \leq \frac{1}{s^{-1}} \bbE_{0,0}^{s,w} \bigg(\cC_{< s}([-v, 0])  \
	; \cB_{<s}\bigg) \times \frac{\bbP_{s, w}^{t, 0} (\cB_{>s})}{t^{-1}}  \\
	& \ \ \ + \frac{1}{t^{-1}} \bbE_{s,w}^{t, 0} \bigg(\cC_{> s}([-v, 0])  \
	; \cB_{>s} \bigg) \times \frac{\bbP_{0,0}^{s, w} (\cB_{<s})}{s^{-1}} \\
	& \ \ \ + \bbE[\cE_{s}^{s} \big([-v, 0] - w \big)|\wh{h}^*_{s} = z] \times \frac{\bbP_{0,0}^{s, w} (\cB_{<s})}{s^{-1}} \times \frac{\bbP_{s, w}^{t, 0} (\cB_{>s})}{t^{-1}} 
\end{align}
By ballot bounds (after using monotonicity to replace $w+z$ by $0$),
\[
\bbP_{0,0}^{s, w} (\cB_{<s}) \lesssim  s^{-1} (w^- + 1), \ \ \ \bbP_{s, w}^{t, 0} (\cB_{>s})  \lesssim t^{-1} (w^- + 1 \wedge e^{-\frac{3}{2}w^+ + \frac{w^2}{t}})
\]
for $s\leq t/2$.  Using Palm calculus, 
\begin{align}
	\bbE_{0,0}^{s,w} \bigg(\cC_{< s}([-v, 0])  \
	; \cB_{<s}\bigg) &= \int_{\eta v^2}^{(s'\wedge \eta^{-1}) v^2} \bbE_{0,0}^{s,w} [\cE_{\sigma}^{\sigma} \big([-v, 0] - \wh{W}_{t,\sigma} \big) 1_{\wh{W}_{t,\sigma} \leq M}; \cB_{<s}] \, 2 \rmd \sigma \label{eq:moment_C<s_B<s} \\
	\bbE_{s,w}^{t, 0} \bigg(\cC_{> s}([-v, 0])  \
	; \cB_{>s} \bigg) &= \int_{(s' \vee \eta) v^2}^{\eta^{-1} v^2} \bbE_{s,w}^{t, 0} [\cE_{\sigma}^{\sigma} \big([-v, 0] - \wh{W}_{t,\sigma} \big) 1_{\wh{W}_{t,\sigma} \leq M}; \cB_{>s}] \, 2 \rmd \sigma \label{eq:moment_C>s_B>s}
\end{align}

The expectations in these integrals are estimated using the following:

\begin{lem}\label{cor:CHL_Lemma5.5_adapted} Assume $-v\leq w+z\leq 0$. For $s\geq \sigma \geq \eta v^2 \geq 10r$,
	\begin{align}\label{eq:CHL_Lemma5.5_adapted_0_to_r_to_s}
		& \bbE_{0, 0}^{s, w} \big[\cE_\sigma^\sigma([-v, \infty) - \wh{W}_{t, \sigma}) 1_{\wh{W}_{t, \sigma} \leq M} ; \cA_t(w+z; r, s) \cap \cA_t(0, s)\big] \\
		& \lesssim_{M, \eta} \, (v+1) e^{\sqrt{2}(v+w+z)} (z^+ + 1)  \frac{1}{\sigma (1+s-\sigma) } \frac{1}{\sqrt{\sigma \wedge (s-\sigma)}} e^{\frac{r}{s}\frac{w^2}{2s} }.
	\end{align}
	For $s\leq \sigma \leq \eta^{-1} v^2 \leq t/10$,
	\begin{align}\label{eq:CHL_Lemma5.5_adapted_s_to_t/2_to_t}
		& \bbE_{s, w} ^{t, 0}\big[\cE_\sigma^\sigma([-v, \infty) - \wh{W}_{t, \sigma}) 1_{\wh{W}_{t, \sigma} \leq M} ; \cA_t(w+z; s, t/2) \cap \cA_t(s, t)\big] \\
		& \lesssim_{M, \eta} \, (v+1) e^{\sqrt{2}(v+w+z)} (z^+ + 1)  \frac{1}{ (1+\sigma-s) t } \frac{1}{\sqrt{\sigma-s}} e^{100 \frac{s}{t}\frac{w^2}{2s} }.
	\end{align}
\end{lem}

Using this we can conclude

\begin{proof}[Proof of Lemma~\ref{lem:HM_eta_approximation}] Applying the estimate \eqref{eq:CHL_Lemma5.5_adapted_0_to_r_to_s} in equation \eqref{eq:moment_C<s_B<s}, we have
\begin{align}
	\bbE_{0,0}^{s,w} \bigg(\cC_{< s}([-v, 0])  \
	; \cB_{<s}\bigg) &\lesssim_{M, \eta} (v+1) e^{\sqrt{2}(v+w+z)} (z^+ + 1)   e^{\frac{r}{s}\frac{w^2}{2s} } \\
	& \qquad \quad \times \int_{\eta v^2}^{(s'\wedge \eta^{-1}) v^2} \frac{1}{\sigma (1+s-\sigma) } \frac{1}{\sqrt{\sigma \wedge (s-\sigma)}} d\sigma.
\end{align}
By separating into cases $s'\geq 2\eta^{-1}$, $\eta \leq s' \leq 2\eta^{-1}$, and $s' \leq \eta$, we can check that the last integral is $\lesssim_\eta s^{-1} 1_{s'\geq \eta}$. Similarly, applying the estimate \eqref{eq:CHL_Lemma5.5_adapted_s_to_t/2_to_t} in equation \eqref{eq:moment_C>s_B>s}, we have
\begin{align}
	\bbE_{s,w}^{t,0} \bigg(\cC_{> s}([-v, 0])  \
	; \cB_{>s}\bigg) &\lesssim_{M, \eta} (v+1) e^{\sqrt{2}(v+w+z)} (z^+ + 1)   e^{100 \frac{s}{t}\frac{w^2}{2s} } \\
	& \qquad \quad \times \int_{(s' \vee \eta) v^2}^{\eta^{-1} v^2} \frac{1}{ (1+\sigma-s) t } \frac{1}{\sqrt{\sigma-s}} d\sigma.
\end{align}
By separating into cases $s'\geq \eta/2$ and $s' \leq \eta/2$, we can check that the last integral is $\lesssim_\eta t^{-1} (1_{s'\geq \eta/2} + v^{-1} 1_{s'\leq \eta/2})$. Combining the bounds derived so far, we have
\begin{align}
	F^{M, \eta}_v(z)  &\lesssim x^{-1} e^{\sqrt{2}(w+z)}  (z^+ + 1) (w^- + 1) (1_{s'\geq \eta/2} + v^{-1} 1_{s'\leq \eta/2})\cdot e^{\max(\frac{r}{s}, 100\frac{s}{t}) \frac{w^2}{2s}}\\
	& \ \ \ + x^{-1} v^{-1} e^{-\sqrt{2} v}  1_{s'\geq \eta} \, \bbE[\cE_{s}^{s} \big([-v, 0] - w \big)|\wh{h}^*_{s} = z] (w^- + 1) (w^- + e^{-\frac{3}{2}w^+}) \cdot e^{\frac{w^2}{t}}.
\end{align}
Plugging this into \eqref{eq:H_M_eta_main}, we obtain
\begin{align}
	H^{M, \eta}_{t, v, r}(\tilde x, s', w) &\lesssim x^{-1}  (1_{s'\geq \eta/2} + v^{-1} 1_{\frac{r}{v^2} \leq s'\leq \eta/2}) \times  \int_{z\leq -w} e^{\sqrt{2}(w+z)}  (z^+ + 1) (w^- + 1) \, \bbP\big(\wh{h}^*_{s} \in dz\big) \\
	& \ \ \ + x^{-1} v^{-1} e^{-\sqrt{2} v} 1_{s'\geq \eta} \, \bbE[\cE_{s}^{s} \big([-v, 0] - w \big); \wh{h}^*_{s} \leq -w] ((w^-)^2 + e^{-\frac{3}{2}w^+}).
\end{align}
where the trailing exponential factors have been absorbed by $e^{-\frac{(w+\gamma_{t,s})^2}{2 s}}$. Now, because $r \geq v^{\alpha}$, we can upper bound $v^{-1} $ by $s'^{\frac{1}{2-\alpha}}$ for $\frac{r}{v^2} \leq s'\leq \eta/2$ in the first line. Using the uniform exponential tail of $\wh{h}^*_{s}$, it is easy to check that the integral over $z$ in the first line is at most $C e^{-c |w|}$ for some $c > 0$. By the moment bound of \cite[Lemma 4.2]{CHL19}, the expectation in the second line is at most $C ve^{\sqrt{2}v} e^{\sqrt{2} w^+ - (\sqrt{2}-\frac{1}{2})w^-} (w^- + 1)$. Overall this gives the bound claimed. 
\end{proof}

It remains to verify Lemma \ref{cor:CHL_Lemma5.5_adapted}. For this purpose we need the following version of \cite[Lemma 5.5]{CHL19}: 

\begin{lem}\label{lem:CHL_Lemma5.5_adapted} For $\wt{v}, \wt{M} \geq 0$ and $s_1 \leq \sigma \leq s_2$,
	\begin{align}
		& \bbE_{s_1, x_1}^{s_2, x_2} \big[\cE_\sigma^\sigma([-\wt{v}, \infty) - \wh{W}_{t, \sigma}) 1_{\wh{W}_{t, \sigma} \leq \wt{M}} ; \cA_t(s_1, s_2)\big] \\
		& \lesssim \, (\wt{v}+1) e^{\sqrt{2}\wt{v}} \frac{(x_1^-+1)(x_2^-+1)}{(1+\sigma-s_1)(1+s_2-\sigma)} \frac{1}{\sqrt{(\sigma-s_1)\wedge (s_2-\sigma)}} e^{\frac{(x_1^-+\wt{M})^2}{2(\sigma-s_1)} \wedge \frac{(x_2^-+\wt{M})^2}{2(s_2-\sigma)}}
	\end{align}
\end{lem}

\begin{proof}[Proof of Lemma \ref{lem:CHL_Lemma5.5_adapted}] Conditioning on the time $\sigma$, the left side equals
	\begin{equation}
		\int_{y\leq \wt{M}} \bbP_{s_1, x_1}^{s_2, x_2}(\wh{W}_{t, \sigma} \in dy) \times \bbE[\cE_\sigma([-\wt{v}, \infty) - y); \wh{h}^*_\sigma \leq -y] \times \bbP_{s_1, x_1}^{\sigma, y}(\cA_t(s_1, \sigma)) \times \bbP_{\sigma, y}^{s_2, x_2}(\cA_t(\sigma, s_2)) 
	\end{equation}
	Up to a multiplicative constant, the Gaussian density is bounded by $\frac{1}{\sqrt{(\sigma-s_1)\wedge (s_2-\sigma)}}$, the expectation is bounded by $e^{\sqrt{2}(\wt{v}+y)} (\wt{v}+1)(y^-+1) (e^{-\frac{(\wt{v}+y)^2}{4\sigma}} + e^{-\frac{\wt{v}+y}{2}})$ by \cite[Lemma 4.2]{CHL19}, the third probability is bounded by 
	\begin{equation}
		\frac{x_1^-+1}{1+\sigma-s_1} \times \begin{cases} y^- + 1 & y\leq 0\\
			e^{-\frac{3}{2} y^+} \times e^{\frac{(\wt{M} + x_1^-)^2}{2(\sigma-s_1)}} & 0\leq y \leq \wt{M}
		\end{cases}
	\end{equation}
	from ballot bounds, and the fourth probability is bounded by the same expression with $x_2^-$ and $s_2-\sigma$ in place of $x_1^-$ and $\sigma-s_1$. Integrating over $y \leq \wt{M}$ then yields the overall bound.
\end{proof}

\begin{proof}[Proof of Lemma \ref{cor:CHL_Lemma5.5_adapted}] We only present the details for verifying \eqref{eq:CHL_Lemma5.5_adapted_0_to_r_to_s} since the computations are largely similar for \eqref{eq:CHL_Lemma5.5_adapted_s_to_t/2_to_t}. Conditioning on the time $r$, the left side of \eqref{eq:CHL_Lemma5.5_adapted_0_to_r_to_s} equals
	\begin{equation}
		\int \bbP_{0,0}^{s, w}(\wh{W}_{t,r} \in dy) \times \bbP_{0, 0}^{r, y}(\cA_t(0, r)) \times \bbE_{r, y}^{s, w} \big[\cE_\sigma^\sigma([-v, \infty) - \wh{W}_{t, \sigma}) 1_{\wh{W}_{t, \sigma} \leq \wt{M}} ; \cA_t(w+z; r, s)\big] 
	\end{equation}
	The density of $y$ is Gaussian with mean $\frac{r}{s} w$ and variance $r(1-\frac{r}{s})$. By ballot upper bound, the first probability in the integrand is at most $C (y^-+1)/r$. Applying Lemma \ref{lem:CHL_Lemma5.5_adapted} with $\wt{v} = v+w+z \in [0, v]$ and $\wt{M} = M -(w+z) \in [M, M+v]$, the last expectation is bounded by
	\begin{align}
		& \bbE_{r, y-(w+z)}^{s, -z} \big[\cE_\sigma^\sigma([-\wt{v}, \infty) - \wh{W}_{t, \sigma}) 1_{\wh{W}_{t, \sigma} \leq \wt{M}} ; \cA_t(r, s)\big] \\
		& \lesssim \, (v+1) e^{\sqrt{2}(v+w+z)} \frac{(z^+ +1)(y^-+1)}{(1+\sigma-r)(1+s-\sigma)} \frac{1}{\sqrt{(\sigma-r)\wedge (s-\sigma)}} e^{\frac{(y^-+M+v)^2}{2(\sigma-r)}}.
	\end{align}
	Multiplying everything together and comparing with the right side \eqref{eq:CHL_Lemma5.5_adapted_0_to_r_to_s} (and using $\sigma \geq 10r$), we just need to show that
	\begin{equation}\label{eq:CHL_Lemma5.5_adapted_0_to_r_to_s_final}
		\int \frac{dy }{\sqrt{2\pi r}} e^{-\frac{(y-\frac{r}{s}w)^2}{2r}} \frac{(y^- + 1)^2}{r} e^{\frac{(y^-+M+v)^2}{2(\sigma-r)}}   \lesssim_{M, \eta} e^{\frac{r}{s}\frac{w^2}{2s} }
	\end{equation}
	Applying the elementary inequalities $(A-B)^2 \geq A^2/2 - B^2$ and $(A+B)^2 \leq 2A^2+2B^2$ to the two exponents, this integral is at most
	\begin{equation}
		\int \frac{dy }{\sqrt{2\pi r}} e^{-\frac{y^2}{4r}(1-\frac{4r}{\sigma-r})} \frac{(y^- + 1)^2}{r}  \times e^{\frac{(\frac{r}{s} w)^2}{2r} }e^{\frac{(M+v)^2}{\sigma-r} }   
	\end{equation}
	The bound \eqref{eq:CHL_Lemma5.5_adapted_0_to_r_to_s_final} now follows from the condition $\sigma \geq \eta v^2 \geq 10r$. 
\end{proof}

%% file: critical.tex
In order to prove Proposition~\ref{prop:6.5}, we shall need the following moment estimates for the level sets of the extremal process at finite time $t > 0$. Similar estimates, holding (at least) up to levels $\Theta(t^{1/2+\epsilon})$ below $0$, were given in~\cite{CHL19}, and we only need slight modifications of which. For the first moment we have
\begin{lem}[Lemma~4.2 in~\cite{CHL19}]
\label{l:18}
There exists $g: \bbR \to \bbR$ such that $g(u) \sim u$ as $u \to \infty$ such that as $t \to \infty$,
\begin{equation}
\label{e:2002}
\bbE \Big(\cE_t\big([v, \infty)\big) ;\; \wh{h}^*_t \leq u \Big)
= \frac{1}{\sqrt{\pi}}
\rme^{-\sqrt{2} v - \frac{v^2}{2t}} (u-v) g(u) \big(1+o(1))
\end{equation}
where $o(1) \to 0$ as $t \to \infty$ uniformly in $u,v$ satisfying $u > -1/\epsilon$, $v < -t^{\epsilon}$ and $(u^++1)(u-v) \leq t^{1-\epsilon}$
 for any fixed $\epsilon > 0$.
\end{lem}
\begin{proof}
This is essentially the second part of~\cite[Lemma~4.2]{CHL19}. There, $u$ is restricted to being at most $1/\epsilon$, but the proof also works verbatim for all $u$ as in the current statement, thanks to the uniformity of the asymptotics in the second part of~\cite[Lemma~3.4]{CHL19}.
\end{proof}

Turning to the second moment, here we shall need a slightly different version from that given in~\cite{CHL19}. It is sharper for $u-v = O(\sqrt{t})$, as the linear term in $u^+$ is now sublinear.
\begin{lem}
\label{l:6.4}
There exists $C > 0$ such that for all $v \leq u \wedge 0$, $u \geq 0$,
\begin{equation}
\bbE \Big(\cE_t\big([v, \infty)\big)^2 ;\; \wh{h}^*_t \leq u \Big) 
\leq \rme^{-2\sqrt{2} v} (u-v+1)^2 \rme^{\sqrt{2}u} (u^++1)^{3/4}
\end{equation}
\end{lem}

\begin{proof}
The proof is a slight modification of the proof of Lemmas~4.3 and~4.4 in~\cite{CHL19}, and we shall thus just highlight the needed changes. In the proof of \cite[Lemma~4.3]{CHL19}, for $r$ such that $r \wedge (t-r) \leq u/4$ we shall use a different bound on the integral
\begin{equation}
\label{e:502}
\int_z \wt{\bbP} \Big( \wh{h}_{t-r}(X_{t-r}) - m_{t,r} \in \rmd z, \,
		   		   \wh{h}^*_t \big(\rmB_{r}(X_t)^\rmc \big) \leq u \Big)   
\times \wt{\bbP}\big( \wh{h}_r(X_r) \geq v-z \,, \wh{h}_r^* \leq u-z \big)^2 \,.
\end{equation} 
For the first term in the integrand, we reintroduce the quadratic exponent in $z$ into the density of $\wh{h}_{t-r}(X_{t-r}) - m_{t,r}$ in the bound for~\cite[(4.18)]{CHL19}, while keeping the same ballot upper bound used there for~\cite[(4.17)]{CHL19}, to get the combined bound
\begin{equation}
C\big(1+(r \wedge (t-r))\big)^{-3/2} \rme^{-(t-r)}
(u^+ + 1)\big((u-z)^+ + 1\big)
\rme^{- (\sqrt{2}-\eps_{t,r})z -  \frac{z^2}{2(t-r)}} \,,
\end{equation}
where $\eps_{t,r} = o(1)$ as $t-r\to \infty$. For the second term in the integrand, we can replace every occurrence of $t$ by $r$ in the bounds~\cite[(4.19)-(4.20)]{CHL19} to arrive at the combined bound
\begin{equation}
\label{e:677}
C \rme^{-2r} (u-v+1)^2 \rme^{-2\sqrt{2} v} ((u-z)^++1)^2 \rme^{2\sqrt{2}z} \rme^{-\frac{3}{2}(u-z)^-} 
\big(\rme^{-\frac{(v-z)^2}{4r}} + \rme^{-\frac{(v-z)^-}{2}}\big)
  \,.
\end{equation}
Together, the integral in~\eqref{e:502} is at most
\begin{multline}
C\big(1+(r \wedge (t-r))^{-3/2}\big) \rme^{-t-r}  (u-v+1)^2 \rme^{-2 \sqrt{2} v}  \\
\times
\int (u^+ +1) \big((u-z)^++1\big)^3 \rme^{(\sqrt{2}+\eps_{t,r})z} \times \rme^{- \frac{z^2}{2(t-r)}} \times 
\big(\rme^{-\frac{z^2}{4r}} + \rme^{-\frac{z^+}{2}-\frac{3}{2}(u-z)^-}\big)
 \rmd z \,.
\end{multline}
Now, let us first assume $t-r \geq C$ for some constant $C$ so that $\eps_{t,r} \leq 0.1$. If $t-r \leq u/4$, we drop the last factor in the integrand. Optimizing the quadratic exponent, the resulting integral is then bounded by $C (u+1)^4 \rme^{1.1^2 (t-r)} \leq C' \rme^{u}$. On the other hand, if $r \leq u/4$, we drop the middle term and similarly obtain the bound $C (u+1)^4 (\rme^{1.1^2 \times 2r} + \rme^{(\sqrt{2} +0.1 - 1/2)u}) \leq C' \rme^u$. Together with \cite[Lemma~4.3]{CHL19} applied to $r \wedge (t-r) > u/4$, this gives an upper bound of
\begin{multline}
\wt{\bbP}^{(2)} \Big( \min \big\{ \wh{h}_t(X_t(1)) , \wh{h}_t(X_t(2)) \big\} \geq v \,,\,\,
		   \wh{h}_t^* \leq u \,\big|\, \rmd(X_t(1), X_t(2)) = r \Big) \\ 
\leq C \frac{\rme^{-t-r}}{1+\big(r \wedge (t-r) \big)^{5/4}} \rme^{\sqrt{2}u}  (u+1)^{3/4}
	\rme^{-2\sqrt{2} v} (u-v+1)^2 
\end{multline}
for all $r \in [0,t-C]$. Here, the factor $(u+1)^{3/4}$ only really appears in the intermediate range $r \wedge (t-r) > u/4$, where the linear growth in $u^+ + 1$ has been reduced at the cost of a slower decay in denominator. Lastly, to deal with $r \in [t-C, t]$, we can always bound the ballot probability~\cite[(4.17)]{CHL19} by $1$, while replacing the $(t-r)$ factor in~\cite[(4.18)]{CHL19} by $(t-r)^{-1/2}$ and otherwise proceed just as in~\cite[Lemma~4.3]{CHL19}. The net effect is the last bound above needs to be multiplied by $(t-r)^{-1/2}$ for $r \in [t-C, t]$. Altogether, using this estimate in the proof of~\cite[Lemma~4.4]{CHL19} instead of the original bound then yields the desired claim.
\end{proof}

\begin{proof}[Proof of Proposition~\ref{prop:6.5}]
For $s \geq 0$, $u \geq v$, set
\begin{equation}
	R_t(v,u) := \frac{\cE_t\big([v, \infty)\big)}{(u-v+1) e^{-\sqrt{2}v - \frac{v^2}{2t}}} \mathbbm 1_{\{\wh{h}_t^* \leq u\}} \,.
\end{equation}
Thanks to~\eqref{e:N7}, and, in particular, the tightness of the centered maximum, it will be sufficient to show that
\begin{equation}
\label{e:7.5}
\Big(R_t(-v,u) - \frac{1}{\sqrt{\pi}} Z_s 
\Big)\mathbbm 1_{\{\max_{\bbL_s} |h_s(x)| \leq m_s + \log \log s\}}
\,{\overset{\bbP}{\longrightarrow}}\, 0 \,,
\end{equation}
as $t \to \infty$, then $s \to \infty$ and finally $u \to \infty$,  uniformly in $v$ as in the statement of the proposition. Henceforth all the asymptotic statements will be taken under these limits.

We now decompose the atoms in $\cE_t$ according to their ancestor at time $s$, and write $R_t^{(x)}(v,u)$ for the analog of $R_t(v,u)$ only with respect to the BBM rooted at $x$. Then the difference between the left hand side of~\eqref{e:7.5} and
\begin{equation}
\label{e:7.6}
\begin{split}
\sum_{x \in \bbL_s}\Big(R^{(x)}_{t-s} \big(& -v-h_s(x)+m_t - m_{t-s}\,,\,\,
	u-h_s(x)+m_t - m_{t-s}\big)
	\rme^{\sqrt{2}(h_s(x) - \sqrt{2}s)} \big(1+o(1)\big) \\
&	- \frac{1}{\sqrt{\pi}}\big(\sqrt{2}s - h_s(x)\big) \rme^{\sqrt{2} (h_s(x) -\sqrt{2}s)}
\Big)
\mathbbm 1_{\{\max_{\bbL_s} |h_s(x)| \leq m_s + \log \log s\}}
\end{split}
\end{equation}
is not zero only on $\{\wh{h}_t^* > u\}$ and thus tends to $0$ in probability. Above, we have used that $m_t - m_{t-s} = \sqrt{2} s + o(1)$.

Now, by Lemma~\ref{l:18}, on 
\begin{equation}
\label{e:7.14}
\big\{\max_{\bbL_s} |h_s(x)| \leq m_s + \log \log s\big\} \,,
\end{equation}
we have,
\begin{equation}
\bbE\Big(R^{(x)}_{t-s} \big(-v-h_s(x)+m_t - m_{t-s}\,,\,\,
	u-h_s(x)+m_t - m_{t-s}\big)\,\Big|\, \cF_s\Big)
= \frac{1+o(1)}{\sqrt{\pi}} \big(\sqrt{2}s - h_s(x)\big) \,,
\end{equation}
as $m_t - m_{t-s} = \sqrt{2} s + o(1)$ and so $- h_s(x) + m_t - m_{t-s}  = \Omega(\log s)$ on \eqref{e:7.14}. This makes the conditional mean of the sum in~\eqref{e:7.6} under $\bbP(-|\cF_s)$ equal to $o(1) |Z_s|$.
At the same time, by Lemma~\ref{l:6.4}, still on~\eqref{e:7.14}, we have
\begin{multline}
\bbE \Big(\big(R^{(x)}_{t-s} \big( -v-h_s(x)+m_t - m_{t-s}\,,\,\,
	u-h_s(x)+m_t - m_{t-s}\big)\big)^2\,\Big|\, \cF_s\Big) \\
= \big(\sqrt{2}s - h_s(x)\big) \rme^{-\sqrt{2} (h_s(x)-\sqrt{2}s)} o(1) \,.
\end{multline}
Then by independence of the $R^{(x)}_{t-s}$ among $x\in \bbL_s$, the variance of the sum in~\eqref{e:7.6} under $\bbP(-|\cF_s)$ is $o(1) |Z_s|$ as well. By Chebyshev's inequality and tightness of $Z_s$, this implies that for any fixed $\eps > 0$, the $\bbP(-|\cF_s)$-probability of (the absolute value of) the left hand side of~\eqref{e:7.5} being more than $\eps$ converges to zero in the limits taken. Then Bounded Convergence makes this hold for $\bbP$-probabilities as well.
\end{proof}

%% file: ballot.tex

\begin{proof}[Proof of Lemma~\ref{lem:15}]
 This is nearly verbatim from Lemma~3.4 in~\cite{CHL19}, except that we allow the endpoint to be any $R\leq t$. The exact same proof there goes through (i.e.~applying Lemma \ref{l:lisa} with $R-r$ in place of $t-r$ to check the barrier curve condition of~\cite{CHL17_supplement} is satisfied), and we omit the repetitive details.
\end{proof}

Finally, we give,
\begin{proof}[Proof of Lemma \ref{lemma_mixed_ballot}] We first show the estimate \eqref{eq:ballot_prob_0_to_s}. Conditioning on the time $r = r_s$ and using Markov property and the Law of Total Probability, the probability on the left equals
	\begin{equation} \label{eq:ballot_prob_0_to_s_integral}
		\int  \bbP_{0,0}^{r, y} \left(\cA_t(0, r) \right) \bbP_{r,y}^{s,w} \left(\cA_t(w+z; r, s) \right) \bbP_{0, 0}^{s, w}\left(\wh{W}_{t,r} \in \rmd y\right) 
	\end{equation}
	By the ballot upper bound \eqref{e:52} in Lemma \ref{lem:15}, the integral over $y\geq 0$ is only $O(r_s^{-1} s^{-1})$. From now on we restrict to $y < 0$ in the integral above. By direct computation, the law of $\wh{W}_{t,r}$ under $\bbP_{0, 0}^{s, w}$ is Gaussian with mean $O(\log r)$ and variance $r(1-r/s)$. Applying the change of variable $y = \sqrt{r} x$, the density of the rescaled variable $x$ is Gaussian with mean $o(1)$ and variance $\sim 1$ as $s\to \infty$ (recall that $r_s = o(s)$). By the ballot estimates \eqref{e:52} and \eqref{e:53} in Lemma \ref{lem:15}, the first probability in \eqref{eq:ballot_prob_0_to_s_integral} is
	\[
	\bbP_{0,0}^{r, \sqrt{r} x} \left(\cA_t(0, r) \right) \sim_r \frac{2f^{(0)}(0) g(\sqrt{r} x)}{r} \sim_r 2f^{(0)}(0) \, \frac{x^-}{\sqrt{r}}
	\]
	and is bounded by $C (x^- + 1)/\sqrt{r}$. For the second probability in \eqref{eq:ballot_prob_0_to_s_integral}, \eqref{e:52} gives an upper bound $C(w, z) \sqrt{r} (x^- + 1)/(s-r)$ for some constant $C(w, z)$ depending on $w, z$. To get an asymptotic for this probability, we follow the original proof of Lemma \ref{lem:15} in \cite[Lemma 3.4]{CHL19}. By translating in time and space, we can rewrite this probability as
\begin{align}
		& \bbP_{r, \sqrt{r}x-w-z}^{s, -z} \left(\cA_t(0; r, s) \right) \\
        = & \bbP \Big (
		\max_{\sigma_k \in [0,s-r]} \big(W_{\sigma_k} - \gamma^{(r)}_{s-r, \sigma_k} - Y^{(r)}_{\sigma_k}\big) \leq 0 \, \Big| W_0 = \sqrt{r} x-w-z, W_{s-r} = -z \Big) \label{eq:ballot_with_variable_starting_point}
	\end{align}
	where $\gamma^{(r)}_{s-r, \sigma} := \log^+(r+\sigma) - \Big(\frac{s-r-\sigma}{s-r} \log^+ r + \frac{\sigma}{s-r} \log^+ s \Big)$ and $Y^{(r)}_{\sigma} := -\wh{h}^{(r+\sigma)*}_{r + \sigma}$. Keeping in mind that $r = r_s$ and $s-r = s-r_s$ are uniquely determined by $s$, a slight adaptation of results of \cite{CHL17_supplement} (allowing decorations $Y$ to also depend on the parameter $t$ there like in the curve $\gamma_{t, \cdot}$) then gives the asymptotic 
	\[
	\frac{2f(\sqrt{r} x-w-z)g(-z)}{s-r} \sim_s 2 g(-z) \frac{\sqrt{r} x^- }{s}
	\]
	for the probability \eqref{eq:ballot_with_variable_starting_point} as $s\to \infty$. Altogether, we have by Dominated Convergence the integral \eqref{eq:ballot_prob_0_to_s_integral} is asymptotic to
	\begin{equation}
		\frac{2f^{(0)}(0) \cdot 2 g(-z)}{s} \int_{x < 0}  (x^-)^2 \frac{1}{\sqrt{2\pi}} e^{-x^2/2} dx = \frac{2f^{(0)}(0) g(-z)}{s}
	\end{equation}
	as $s\to \infty$. 
	
	The computation of the estimate \eqref{eq:ballot_prob_s_to_t} is largely similar. Conditioning on the time $t/2$, this probability equals
	\begin{equation} \label{eq:ballot_prob_s_to_t_integral}
		\int  \bbP_{s,w}^{t/2,y} \left(\cA_t(w+z; s, t/2) \right) \bbP_{t/2,y}^{t,0} \left(\cA_t(t/2, t) \right) \bbP_{s, w}^{t, 0}\left(\wh{W}_{t,t/2} \in \rmd y\right)
	\end{equation}
	Like before, we may restrict the integral to $y < 0$. This time, we apply the change of variable $y = \sqrt{t} x/2$, so that the density of the rescaled variable $x$ is Gaussian with mean $o(1)$ and variance $\sim 1$ as $t \to \infty$ for fixed $s$. For fixed $s$, applying the ballot estimates in Lemma \ref{lem:15} (after translating in space), the first probability in \eqref{eq:ballot_prob_s_to_t_integral} is
	\[
	\bbP_{s,-z}^{t/2, \sqrt{t} x/2 -w-z} \left(\cA_t(0; s, t/2) \right) \sim_t \frac{2f^{(s)}(-z) g(\sqrt{t} x/2 -w-z)}{t/2} \sim_{t} 2f^{(s)}(-z) \frac{x^-}{\sqrt{t}}
	\]
	and is bounded by $C(w, z) (x^- + 1)/\sqrt{t}$ for some constant $C(w, z)$ depending on $w, z$. For the second probability in \eqref{eq:ballot_prob_s_to_t_integral}, \eqref{e:52} gives an upper bound $C (x^- + 1)/\sqrt{t}$ and a similar calculation as in \eqref{eq:ballot_with_variable_starting_point} gives the asymptotic
	\[
	\bbP_{t/2,\sqrt{t}x/2 }^{t,0} \left(\cA_t(t/2, t) \right)  \sim_t \frac{2 f(\sqrt{t} x/2) g(0)}{t/2} \sim_t 2 g(0) \frac{x^-}{\sqrt{t}}.
	\]
	Altogether, we have by Dominated Convergence that the integral \eqref{eq:ballot_prob_s_to_t_integral} is asymptotic to
	\begin{equation}
		\frac{2f^{(s)}(-z) \cdot 2 g(0)}{t} \int_{x < 0}  (x^-)^2 \frac{1}{\sqrt{2\pi}} e^{-x^2/2} dx = \frac{2f^{(s)}(-z) g(0)}{t}
	\end{equation}
	as $t\to \infty$ for fixed $s$. Finally, taking $s\to \infty$ and apply the asymptotic \eqref{e:54} we establish \eqref{eq:ballot_prob_s_to_t}. 
\end{proof}

\begin{proof}[Proof of Lemma~\ref{l:10.7.0}]
Lemma~2.7 in~\cite{CHL17_supplement}, specialized to $W, (-\wh{h}^{s*}_s)_{s \geq 0}, \cN, \gamma$ as $W, Y, \cN, \gamma$ immediately gives
\begin{equation}
\bbP \Big( \max_{s\in[s,\, t/2]} \wh{W}_{t,u} > -M,\, \cA_t \,\Big|\, \wh{W}_{t,0} = 0 \,,\,\, \wh{W}_{t,t} = 0\Big)
	\leq C \frac{(M+1)^2}{t\sqrt{s}} \,,
\end{equation}
after noting that $\cQ_t(0,t)$ identifies with $\cA_t$ under the above specialization. Together with the lower bound,
\begin{equation}
\bbP \big( \cA_t \,\big|\, \wh{W}_{t,0} = 0 \,,\,\, \wh{W}_{t,t} = 0\big)
	\geq c \frac{1}{t} \,,
\end{equation}
which is implied by Lemma~\ref{lem:15}, this gives the desired statement.
\end{proof}

\begin{proof}[Proof of Lemma~\ref{l:2.13}]
The proof of this exact statement appears in the beginning of the proof of Lemma~2.8 in~\cite{CHL17_supplement}, specialized to $W, (-\wh{h}^{s*}_s)_{s \geq 0}, \cN, \gamma$ as $W, Y, \cN, \gamma$.
\end{proof}